\documentclass{amsart}

\usepackage{amssymb}
\usepackage{overpic}
\usepackage{graphics}
\usepackage{epsfig}
\usepackage{amsmath}
\usepackage{amsfonts}
\usepackage{paralist}

\newtheorem{theorem}{Theorem}[section]
\newtheorem{corollary}{Corollary}

\newtheorem{lemma}[theorem]{Lemma}

\theoremstyle{definition}
\newtheorem{definition}[theorem]{Definition}
\newtheorem{remark}{Remark}

\begin{document}

\title[Isospectral Graph Reductions and Improved Estimates]{Isospectral Graph Reductions and Improved Estimates of Matrices' Spectra}

\author{L. A. Bunimovich}
\address{ABC Math Program and School of Mathematics, Georgia Institute of Technology, 686 Cherry Street, Atlanta, GA 30332}
\email{bunimovich@math.gatech.edu}

\author{B. Z. Webb}
\address{Brigham Young University, Department of Mathematics, Provo, UT 84602}
\email{bwebb@math.byu.edu\\
Telephone: 801-422-2336}



\maketitle


\begin{abstract}
Via the process of isospectral graph reduction the adjacency matrix of a graph can be reduced to a smaller matrix while its spectrum is preserved up to some known set. It is then possible to estimate the spectrum of the original matrix by considering Gershgorin-type estimates associated with the reduced matrix. The main result of this paper is that eigenvalue estimates associated with Gershgorin, Brauer, Brualdi, and Varga improve as the matrix size is reduced. Moreover, given that such estimates improve with each successive reduction, it is also possible to estimate the eigenvalues of a matrix with increasing accuracy by repeated use of this process.
\end{abstract}

\section{Introduction}
A remarkable theorem due to Gershgorin \cite{Gershgorin31} states that if the matrix $A\in\mathbb{C}^{n\times n}$ then the eigenvalues of $A$ are contained in the union of the $n$ discs $$\bigcup_{i=1}^n\{\lambda\in\mathbb{C}:|\lambda-A_{ii}|\leq\sum_{j=1, j\neq i}^n|A_{ij}|\}.$$ This simple and geometrically intuitive result moreover implies a nonsingularity result for diagonally dominant matrices (see theorem 1.4 in \cite{Varga09}), which can be traced back to earlier work done by L\'{e}vy, Desplanques, Minkowski, and Hadamard \cite{Levy81,Desplanques87,Minkowski00,Hadamard03}. More recently, this result of Gershgorin has been improved by Brauer and Varga \cite{Brauer47,Varga09} whose results are similar in spirit to Gershgorin's in that each assigns to every matrix $A\in\mathbb{C}^{n\times n}$ a region of the complex plane containing the matrix's eigenvalues. Moreover, the same holds for a result of Brualdi \cite{Brualdi82} with the exception that the associated region is define for a proper subset of $\mathbb{C}^{n\times n}$.

These improvements can be summarized as follows. If $A\in\mathbb{C}^{n\times n}$ let $\Gamma(A)$, $\mathcal{K}(A)$, and $B(A)$ denote the associated regions given respectively by Gershgorin, Brauer, and the improvement of Brualdi's theorem given by Varga. If $\sigma(A)$ denotes the eigenvalues of $A$ then it is known that $\sigma(A)\subseteq B(A)\subseteq\mathcal{K}(A)\subseteq\Gamma(A)$ for any complex valued matrix $A$ (see \cite{Varga09} for details). Furthermore, if the region $br(A)$ associated with Brualdi's original result is defined then it follows that $\sigma(A)\subseteq B(A)\subseteq br(A)\subseteq \mathcal{K}(A)$.

In this paper, our goal is to improve each of the estimates of Gershgorin, Brauer, Brualdi, and Varga by considering reductions in the structure of the weighted digraphs associated to each matrix $A\in\mathbb{C}^{n\times n}$. To do so we first extend these classical results to a larger class of square matrices with entries in the set $\mathbb{W}$ consisting of complex rational functions. The motivation for considering the class of matrices with entries in $\mathbb{W}$ arises from the following.

In the study of dynamical networks, in which networks are typically described by large and often complex graphs of interactions, it has been found that an important characteristic of a network is the spectrum of the network's adjacency matrix \cite{Blank06,Newman06,Afriamovich07,Motter07}. Using the theory developed in \cite{BW09} it is possible to reduce the graph $G$ associated to some network to another smaller graph $\mathcal{R}$. We refer to this reduction process as an \textit{isospectral graph reduction}, or simply a graph reduction, of $G$.

The main result of \cite{BW09} is that the eigenvalues of the adjacency matrix $M(G)$ of $G$ and the adjacency matrix $M(\mathcal{R})$ of $\mathcal{R}$ differ at most by some set, which is known in advance. What is novel about this process is that it equivalently allows for the reduction of an arbitrary matrix $A\in\mathbb{C}^{n\times n}$ to a smaller matrix $R\in\mathbb{W}^{m\times m}$ ($m<n$) such that the eigenvalues of $A$ and $R$ differ by at most some set, known in advance.

In the present paper we show that by using such graph reductions (equivalently matrix reductions) one can improve Gershgorin, Brauer, and Brualdi-type estimates of the spectra of matrices in $\mathbb{C}^{n\times n}$. Specifically, for $M(G)\in\mathbb{C}^{n\times n}$ the regions in the complex plane for both Gershgorin and Brauer estimates of the eigenvalues of $M(G)$ shrink as the graph $G$ is reduced (see theorems \ref{impgersh} and \ref{impbrauer} for exact statements). For the estimates associated with Brualdi and Varga we give a sufficient condition under which such estimates also improve as the underlying graph is reduced (see theorems \ref{impbrualdi} and \ref{bru}).

We also note that, for a given graph (equivalently matrix), many graph reductions are typically possible. Hence, this process is quite flexible. Moreover, as it is possible to sequentially reduce a graph $G$, graph reductions on $G$ can be used to estimate the spectrum of $M(G)$ with increasing accuracy depending on the extent to which $G$ is reduced. In particular, if $G$ is reduced as much as possible the corresponding Gershgorin region is a finite set of points in the complex plane that differs from the actual spectrum of $M(G)$ by a uniquely defined set of points.

This paper is organized as follows. Section 2 introduces the formal definitions used in this paper. Section 3 extends the results of Gershgorin, Brauer, Brualdi, and Varga to the class of matrices with entries in $\mathbb{W}$. Section 4 then summarizes and expands the theory of isospectral graph reductions developed in \cite{BW09} which will be used to improve the eigenvalue estimates of section 3. Section 5 contains the main results of this paper demonstrating that the procedure of isospectral graph reduction gives better estimates of the spectra of matrices than the aforementioned methods. Section 6 gives some natural applications of the theorems of section 5. These include estimating the spectrum of a Laplacian matrix of graph, estimating the spectral radius of a matrix, and determining useful reductions to use for a given matrix (or equivalently, graph of a network).

\section{Preliminaries}
In this paper we consider two equivalent mathematical objects. The first is the set of graphs consisting of all finite weighted digraphs with or without loops having no parallel edges and edge weights in the set $\mathbb{W}$ of complex rational functions (described below). We denote this class of graphs by $\mathbb{G}$ where $\mathbb{G}^n$ is the set of graphs in $\mathbb{G}$ having $n$ vertices. The second set of objects we consider are the weighted adjacency matrices associated with the graphs in $\mathbb{G}$. That is, the class of matrices $\mathbb{W}^{n\times n}$ for all $n\geq 1$.

By way of notation we let the weighted digraph $G\in\mathbb{G}$ be the triple $(V,E,\omega)$ where $V$ and $E$ are the finite sets denoting the \textit{vertices} and \textit{edges} of $G$ respectively, the edges corresponding to ordered pairs $(v,w)$ for $v,w\in V$. Furthermore, $\omega:E\rightarrow\mathbb{W}$ where $\omega(e)$ is the \textit{weight} of the edge $e$ for $e\in E$. We will use the convention that $\omega(e)=0$ if and only if $e\notin E$.

For convenience, any graph that is denoted by some triple, e.g. $G=(V,E,\omega)$, will be assumed to be in $\mathbb{G}$. Moreover, if the vertex set of the graph $G=(V,E,\omega)$ is labeled $V=\{v_1,\dots,v_n\}$ then we denote the edge $(v_i,v_j)$ by $e_{ij}$. For convenience, in the remainder of this paper if $G=(V,E,\omega)$ is a graph in $\mathbb{G}^n$ then we will assume that its vertex set has some labeling $V=\{v_1,\dots,v_n\}$.

In order to describe the set of weights $\mathbb{W}$ let $\mathbb{C}[\lambda]$ denote the set of polynomials in the single complex variable $\lambda$ with complex coefficients. We define the set $\mathbb{W}$ to be the set of rational functions of the form $p/q$ where $p,q\in\mathbb{C}[\lambda]$ such that $p$ and $q$ have no common factors and $q$ is nonzero.

The set $\mathbb{W}$ is then a field under addition and multiplication with the convention that common factors are removed when two elements are combined. That is, if $p/q,r/s\in\mathbb{W}$ then $p/q+r/s=(ps+rq)/qs$ where the common factors of $ps+rq$ and $qs$ are removed. Similarly, in the product $pr/qs$ of $p/q$ and $r/s$ the common factors of $pr$ and $qs$ are removed.

In order to stress the generality of considering the set $\mathbb{G}$ we note that graphs which are either undirected, unweighted or have parallel edges, can be considered to be graphs in $\mathbb{G}$. This is done by making an undirected graph $G$ into a directed graph by orienting each of its edges in both directions. In the case that $G$ is unweighted, $G$ can weighted by giving each edge unit weight. Also multiple edges between two vertices of $G$ may be considered a single edge by adding the weights of those multiple edges and setting this to be the weight of this single equivalent edge.

To introduce the spectrum associated to a graph $G\in\mathbb{G}$ we will use the following notation. If $G=(V,E,\omega)$ then the matrix $M(G)=M(G,\lambda)$ defined entrywise by $$M(G)_{ij}=\omega(e_{ij})$$
is the \textit{weighted adjacency matrix} of $G$.

We let the \textit{spectrum} or \emph{eigenvalues} of a matrix $M=M(\lambda)\in\mathbb{W}^{n\times n}$ be the set
\begin{equation}\label{eq1}
\{\lambda\in\mathbb{C}:\det(M(\lambda)-\lambda I)=0\}
\end{equation}
where this set includes multiplicities. More specifically, as $$\det(M(\lambda)-\lambda I)=p/q\in\mathbb{W}$$ then the spectrum of $M$ is the solutions to $p=0$.

For the graph $G$ we let $\sigma(G)$ denote the spectrum of $M(G)$. The spectrum of a matrix with entries in $\mathbb{W}$ is therefore a generalization of the spectrum of a matrix with complex entries.

As we are mainly concerned with the properties of the adjacency matrix of graphs in $\mathbb{G}$ we note that there is a one-to-one correspondence between the graphs in $\mathbb{G}^n$ and the matrices $\mathbb{W}^{n\times n}$. Therefore, we may talk of a graph $G\in\mathbb{G}^n$ associated with a matrix $M=M(G)$ in $\mathbb{W}^{n\times n}$ and vice-versa without ambiguity.

\section{Spectra Estimation of Graphs in $\mathbb{G}$.}
Here we extend the classical results of Gershgorin, Brauer, Brualdi, and the more recent work of Varga to matrices in $\mathbb{W}^{n\times n}$ (see for instance \cite{Varga09}). To do so we will first define the notion of a \textit{polynomial extension} of a graph $G\in\mathbb{G}$.

\begin{definition}
If $G\in\mathbb{G}^n$ and $M(G)_{ij}=p_{ij}/q_{ij}$ where $p_{ij},q_{ij}\in\mathbb{C}[\lambda]$ let $L_i(G)$=$\prod_{j=1}^{n}q_{ij}$ for $1\leq i\leq n$. We call the graph $\bar{G}$ with adjacency matrix
$$M(\bar{G})_{ij}=\begin{cases}
                         L_i(G)M(G)_{ij} \hspace{.85in} i\neq j\\
                         L_i(G)\big(M(G)_{ij}-\lambda\big)+\lambda \hspace{0.2in} i=j
                         \end{cases}, \ \ 1\leq i,j\leq n
$$ the \textit{polynomial extension} of $G$.
\end{definition}

\noindent To justify this name note that each $M(\bar{G})_{ij}$ is an element of $\mathbb{C}[\lambda]$ or $M(\bar{G})$ has complex polynomial entries. Moreover, we have the following result.

\begin{lemma}\label{lemma1}
If $G\in\mathbb{G}$ then $\sigma(G)\subseteq\sigma(\bar{G})$.
\end{lemma}

\begin{proof}
For $G\in\mathbb{G}^n$ note that the matrix $M(\bar{G})-\lambda I$ is given by
$$(M(\bar{G})-\lambda I)_{ij}=\begin{cases}
                         L_i(G)M(G)_{ij} \hspace{0.6in} i\neq j\\
                         L_i(G)\big(M(G)_{ij}-\lambda\big) \hspace{0.2in} i=j
                         \end{cases} \text{for} \ 1\leq i\leq n.
$$ The matrix $M(\bar{G})-\lambda I$ is then the matrix $M(G)-\lambda I$ whose $i$th row has been multiplied by $L_i(G)$. Therefore,
$$\det\big(M(\bar{G})-\lambda I\big)=\Big(\prod_{i=1}^n L_i(G)\Big)\det\big(M(G)-\lambda I\big)$$ implying $\sigma(G)\subseteq\sigma(\bar{G})$.
\end{proof}

\subsection{Gershgorin-Type Regions}
As previously mentioned, a theorem of Gershgorin's, originating from \cite{Gershgorin31}, gives a simple method for bounding the eigenvalues of a square matrix with complex valued entries. This result is the following theorem which we formulate after introducing some notation.

If $A\in\mathbb{C}^{n\times n}$ let
\begin{equation}\label{eq-1}
r_i(A)=\sum_{j=1, \ j \neq i}^n|A_{ij}|, \ \ \ \ \ \ 1\leq i \leq n
\end{equation} be the \textit{ith row sum} of $A$.

\begin{theorem}{\textbf{(Gershgorin \cite{Gershgorin31})}}\label{theorem1}
Let $A\in\mathbb{C}^{n\times n}$. Then all eigenvalues of $A$ are contained in the set
$$\Gamma(A)=\bigcup^n_{i=1}\{\lambda\in \mathbb{C}:|\lambda-A_{ii}|\leq r_i(A)\}.$$
\end{theorem}

In order to extend theorem \ref{theorem1} to the class of matrices $\mathbb{W}^{n\times n}$ we use the following adaptation of the notation given by (\ref{eq-1}). For $G\in\mathbb{G}^n$ let
$$r_i(G)=\sum_{j=1,j\neq i}^n|M(G)_{ij}| \ \text{for} \ 1\leq i\leq n$$
be the \textit{ith row sum} of $M(G)$.

Note that as $M(\bar{G})\in\mathbb{C}[\lambda]^{n\times n}$, for any $G\in\mathbb{G}$, we can view $M(\bar{G})=M(\bar{G},\lambda)$ as a function $M(\bar{G},\cdot):\mathbb{C}\rightarrow\mathbb{C}^{n\times n}$ and $M(\bar{G},\cdot)_{ij}:\mathbb{C}\rightarrow\mathbb{C}$. Likewise, we can consider $r_i(\bar{G})=r_i(\bar{G},\lambda)$ to be the function $r_i(\bar{G},\cdot):\mathbb{C}\rightarrow\mathbb{C}$. However, typically we will suppress the dependence of $M(\bar{G})$ and $r_i(\bar{G})$ on $\lambda$ for ease of notation.

\begin{theorem}\label{theorem2}
Let $G\in\mathbb{G}^n$. Then $\sigma(G)$ is contained in the set
$$\mathcal{BW}_\Gamma(G)=\bigcup_{i=1}^n\{\lambda\in\mathbb{C}:|\lambda-M(\bar{G})_{ii}|\leq r_i(\bar{G})\}.$$
\end{theorem}

\begin{proof}
First note that for $\alpha\in\sigma(G)$ the matrix $M(\bar{G},\alpha)\in\mathbb{C}^{n\times n}$. As Lemma \ref{lemma1} implies that $\alpha$ is an eigenvalue of the matrix $M(\bar{G},\alpha)$ then by an application of Gershgorin's theorem the inequality $|\alpha-M(\bar{G},\alpha)_{ii}|\leq r_i(\bar{G},\alpha)$ holds for some $1\leq i\leq n$. Hence, $\alpha\in\mathcal{BW}_\Gamma(G)$.
\end{proof}

Because it will be useful later in comparing different regions in the complex plane, for $G\in\mathbb{G}^n$ we denote
$$\mathcal{BW}_\Gamma(G)_i=\{\lambda\in\mathbb{C}:|\lambda-M(\bar{G})_{ii}|\leq r_i(\bar{G})\} \ \text{where} \ 1\leq i\leq n$$ and call this the \textit{$i$th Gershgorin-type region} of $G$. Similarly, we call the union $\mathcal{BW}_\Gamma(G)$ of these $n$ sets the \textit{Gershgorin-type region} of the graph $G$.

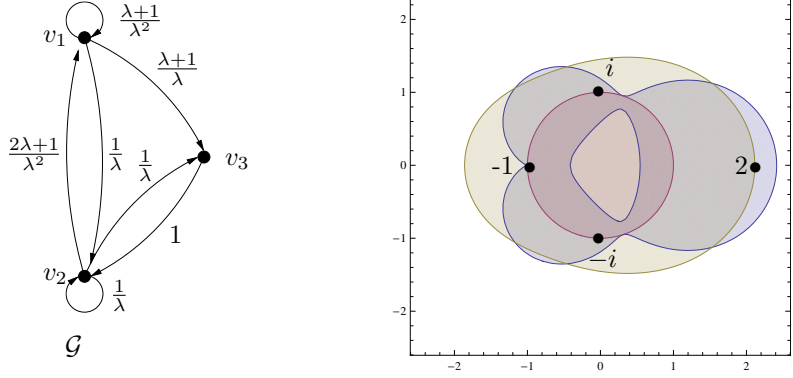
\begin{figure}
\begin{tabular}{ccc}
    \begin{overpic}[scale=.5,angle=270]{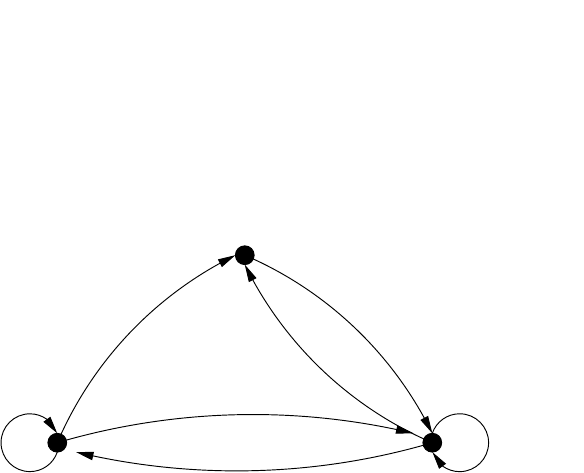}
    \put(0,5){$\mathcal{G}$}
    \put(-6,89){$v_1$}
    \put(43,56){$v_3$ }
    \put(-6,25){$v_2$}
    \put(14,93){$\frac{\lambda+1}{\lambda^2}$ }
    \put(11,57){$\frac{1}{\lambda}$}
    \put(-16,57){$\frac{2\lambda+1}{\lambda^2}$}
    \put(19,54){$\frac{1}{\lambda}$ }
    \put(12,19){$\frac{1}{\lambda}$}
    \put(28,35){$1$ }
    \put(24,80){$\frac{\lambda+1}{\lambda}$}
    \end{overpic} &
    \begin{overpic}[scale=.5]{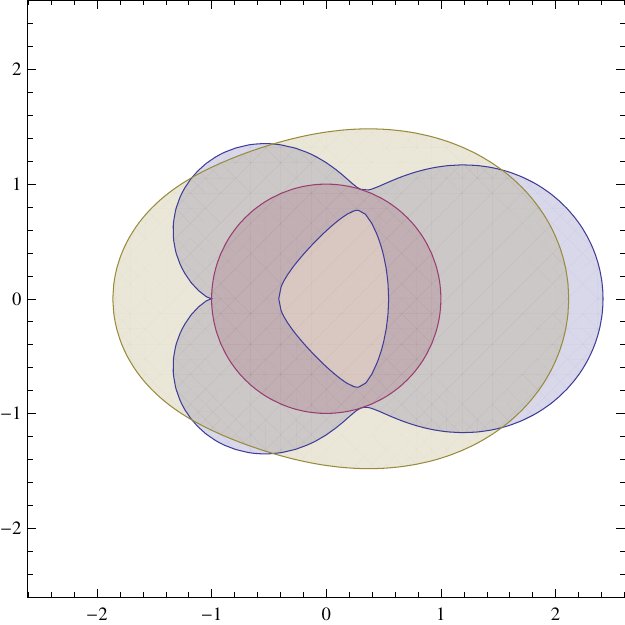}
    \put(25,50){-1 $\bullet$}
    \put(89.5,50){$\bullet$}
    \put(86,50){2}
    \put(50,32){$\bullet$}
    \put(49,26.5){$-i$}
    \put(50,69){$\bullet$}
    \put(53.5,74.5){$i$}
    \end{overpic}
\end{tabular}
\caption{The graph $\mathcal{G}$ (left) and $\mathcal{BW}_\Gamma(\mathcal{G})$ (right) where $\sigma(\mathcal{G})=\{-1,-1,2,-i,i\}$ is indicated.}
\end{figure}

As an illustration of theorem \ref{theorem2} consider the following example. Let $\mathcal{G}\in\mathbb{G}$ be the graph with adjacency matrix
\begin{equation}\label{matrix}
M(\mathcal{G})=\left[ \begin{array}{ccc}
\frac{\lambda+1}{\lambda^2} & \frac{1}{\lambda} & \frac{\lambda+1}{\lambda}\\
\frac{2\lambda+1}{\lambda^2} & \frac{1}{\lambda} & \frac{1}{\lambda}\\
0 & 1 & 0\\
\end{array}
\right].
\end{equation}
As $\det(M(\mathcal{G},\lambda)-\lambda I)=(-\lambda^5+2\lambda^3+2\lambda^2+3\lambda+2)/(\lambda^2)$ one can compute that $\sigma(\mathcal{G})=\{-1,-1,i,-i,2\}$. The corresponding Gershgorin-type region $\mathcal{BW}_\Gamma(\mathcal{G})$ is shown in figure 1 where
\begin{equation*}
M(\bar{\mathcal{G}})=\left[ \begin{array}{ccc}
-\lambda^5+\lambda^3+\lambda^2+\lambda & \lambda^3 & \lambda^4+\lambda^3\\
2\lambda^3+\lambda^2 & -\lambda^5+\lambda^3+\lambda & \lambda^3\\
0 & 1 & 0\\
\end{array}
\right].
\end{equation*}

We note here that $\mathcal{BW}_\Gamma(\mathcal{G})$ is the union of the three regions $\mathcal{BW}_\Gamma(\mathcal{G})_1$, $\mathcal{BW}_\Gamma(\mathcal{G})_2$, and $\mathcal{BW}_\Gamma(\mathcal{G})_3$ whose boundaries are shown in blue, red, and tan. The interior colors of these regions reflect their intersections and the eigenvalues $\sigma(\mathcal{G})$ are indicated as points. In the examples that follow we will use the same technique to display similar regions.

\subsection{Brauer-Type Regions} Following Gershgorin, Brauer was able to give the following eigenvalue inclusion result for matrices with complex valued entries.

\begin{theorem}{\textbf{(Brauer \cite{Varga09})}}\label{theorem3}
Let $A\in\mathbb{C}^{n\times n}$ where $n\geq 2$. Then all eigenvalues of $A$ are located in the set
\begin{equation}\label{eq2}
\mathcal{K}(A)={\bigcup_{\begin{smallmatrix} 1\leq i,j\leq n\\ i\neq j\end{smallmatrix}}}\{\lambda\in \mathbb{C}:|\lambda-A_{ii}||\lambda-A_{jj}|\leq r_i(A)r_j(A)\}.
\end{equation}
\end{theorem}

The individual regions given by $\{\lambda\in \mathbb{C}:|\lambda-A_{ii}||\lambda-A_{jj}|\leq r_i(A)r_j(A)\}$ in equation (\ref{eq2}) are known as Cassini ovals and may consists of one or two distinct components. Moreover, there are $\binom{n}{2}$ such regions for any $n\times n$ matrix with complex entries. As with Gershgorin's theorem we prove an extension to Brauer's theorem for matrices in $\mathbb{W}^{n\times n}$.

\begin{theorem}\label{theorem4}
Let $G\in\mathbb{G}^n$ where $n\geq 2$. Then $\sigma(G)$ is contained in the set
$$\mathcal{BW}_\mathcal{K}(G)=\bigcup_{\begin{smallmatrix} 1\leq i,j\leq n\\ i\neq j\end{smallmatrix}}\{\lambda\in\mathbb{C}:|\lambda-M(\bar{G})_{ii}||\lambda-M(\bar{G})_{jj}|\leq r_i(\bar{G})r_j(\bar{G})\}.$$ Also, $\mathcal{BW}_\mathcal{K}(G)\subseteq\mathcal{BW}_\Gamma(G)$.
\end{theorem}

\begin{proof}
As in the proof of theorem \ref{theorem2}, if $\alpha\in\sigma(G)$ then $\alpha\in\sigma(\bar{G})$ and the matrix $M(\bar{G},\alpha)\in\mathbb{C}^{n\times n}$. Brauer's theorem therefore implies that
$$|\alpha-M(\bar{G},\alpha)_{ii}||\alpha-M(\bar{G},\alpha)_{jj}|\leq r_i(\bar{G},\alpha)r_j(\bar{G},\alpha)$$ for some pair of distinct integers $i$ and $j$. It then follows that, $\alpha\in\mathcal{BW}_\mathcal{K}(G)$ or $\sigma(G)\subseteq\mathcal{BW}_\mathcal{K}(G)$.

Following the proof in \cite{Varga09}, to prove the assertion that $\mathcal{BW}_\mathcal{K}(G)\subseteq\mathcal{BW}_\Gamma(G)$ let
\begin{equation}\label{eq-2} \mathcal{BW}_\mathcal{K}(G)_{ij}=\{\lambda\in\mathbb{C}:|\lambda-M(\bar{G},\lambda)_{ii}||\lambda-M(\bar{G},\lambda)_{jj}|\leq r_i(\bar{G},\lambda)r_j(\bar{G},\lambda)\}
\end{equation} for distinct $i$ and $j$. The claim then is that $\mathcal{BW}_\mathcal{K}(G)_{ij}\subseteq\mathcal{BW}_\Gamma(G)_i\cup\mathcal{BW}_\Gamma(G)_j$. To see this, assume for a fixed $\lambda$ that $\lambda\in\mathcal{BW}_\mathcal{K}(G)_{ij}$ or
$$|\lambda-M(\bar{G},\lambda)_{ii}||\lambda-M(\bar{G},\lambda)_{jj}|\leq r_i(\bar{G},\lambda)r_j(\bar{G},\lambda).$$ If $r_i(\bar{G},\lambda)r_j(\bar{G},\lambda)=0$ then either $\lambda-M(\bar{G},\lambda)_{ii}=0$ or $\lambda-M(\bar{G},\lambda)_{jj}=0$. As $\lambda=M(\bar{G},\lambda)_{ii}$ implies $\lambda\in\mathcal{BW}_\Gamma(G)_i$ and $\lambda=M(\bar{G},\lambda)_{jj}$ implies $\lambda\in\mathcal{BW}_\Gamma(G)_j$ then $\lambda\in\mathcal{BW}_\Gamma(G)_i\cup\mathcal{BW}_\Gamma(G)_j$.

If $r_i(\bar{G},\lambda)r_j(\bar{G},\lambda)>0$ then it follows that
$$\Big(\frac{|\lambda-M(\bar{G},\lambda)_{ii}|}{r_i(\bar{G},\lambda)}\Big)\Big(\frac{|\lambda-M(\bar{G},\lambda)_{jj}|}{r_j(\bar{G},\lambda)}\Big)\leq 1.$$ Since at least one of the two quotients on the left must be less than or equal to 1 then $\lambda\in\mathcal{BW}_\Gamma(G)_i\cup\mathcal{BW}_\Gamma(G)_j$ which verifies the claim and the result follows.
\end{proof}

We call the region $\mathcal{BW}_\mathcal{K}(G)$ the \textit{Brauer-type region} of the graph $G$ and the region $\mathcal{BW}_\mathcal{K}(G)_{ij}$ given in (\ref{eq-2}) the \textit{$ij$th Brauer-type region} of $G$. Using theorem \ref{theorem4} on the graph $\mathcal{G}$ given in figure 1 we have the Brauer-type region shown in the left hand side of figure 2. On the right is a comparison between $\mathcal{BW}_\mathcal{K}(\mathcal{G})$ and $\mathcal{BW}_\Gamma(\mathcal{G})$ where the inclusion $\mathcal{BW}_\mathcal{K}(\mathcal{G})\subseteq\mathcal{BW}_\Gamma(\mathcal{G})$ is demonstrated.

\begin{figure}
\centering
\begin{tabular}{cc}
    \begin{overpic}[scale=.5]{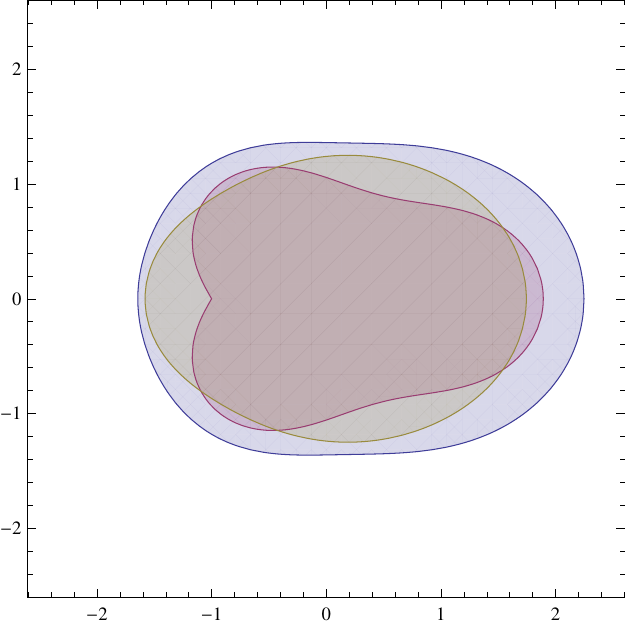}
    \put(25,50){-1 $\bullet$}
    \put(89.5,50){$\bullet$ 2}
    \put(50,32){$\bullet$}
    \put(46,36){$-i$}
    \put(50,69){$\bullet$}
    \put(50.5,64){$i$}
    \end{overpic} &
    \begin{overpic}[scale=.5]{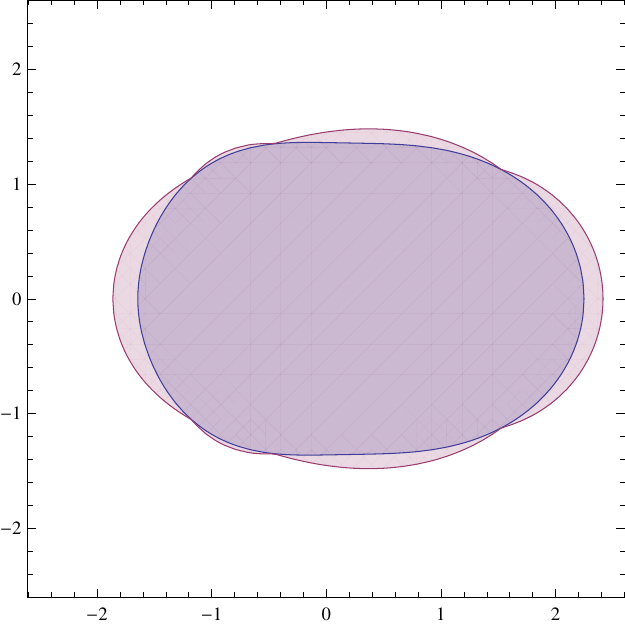}
    \put(25,50){-1 $\bullet$}
    \put(89.5,50){$\bullet$}
    \put(84,50){2}
    \put(51,32){$\bullet$ $-i$}
    \put(51,69){$\bullet$ $i$}
    \end{overpic}
\end{tabular}
\caption{Left: The Brauer region $\mathcal{K}(\mathcal{G})$ for $\mathcal{G}$ in figure 1. Right: $\mathcal{K}(\mathcal{G})\subseteq\Gamma(\mathcal{G})$.}
\end{figure}
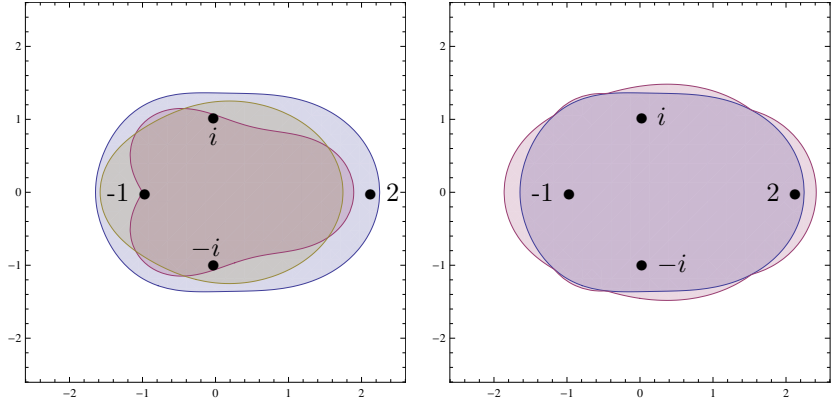

\subsection{Brualdi-Type Regions} In this section we first extend a result of Varga \cite{Varga09}, which is itself an extension of a result of Brualdi \cite{Brualdi82}. This result of Varga relates the spectrum of a graph with complex weights to its cycle structure. We then show that the same can be done for the original result of Brualdi.

A \textit{path} $P$ in the graph $G=(V,E,\omega)$ is a sequence of distinct vertices $v_1,\dots,v_m\in V$ such that $e_{i,i+1}\in E$ for $1\leq i\leq m-1$. In the case that the vertices $v_1,\dots,v_m$ are distinct, with the exception that $v_1=v_m$, then $P$ is a \textit{cycle}. If $\gamma$ is a cycle of $G$ we denote it by its ordered set of vertices. That is, if $e_{i,i+1}\in E$ for $1\leq i\leq m-1$ and $e_{m1}\in E$ then we write this cycle as the ordered set of vertices $\{v_1,\dots,v_m\}$ up to cyclic permutation. Moreover, we call a cycle consisting of a single vertex a \textit{loop}.

A \textit{strong cycle} of $G$ is a cycle $\{v_1,\dots,v_m\}$ such that $m\geq 2$. Furthermore, if $v_i\in V$ has no strong cycle passing through it then we define its associated \textit{weak cycle} as $\{v_i\}$ regardless of whether $e_{ii}\in E$. For $G\in\mathbb{G}$ we let $C_s(G)$ and $C_w(G)$ denote the set of strong and weak cycles of $G$ respectively and let $C(G)=C_s(G)\cup C_w(G)$.

A directed graph is \textit{strongly connected} if there is a path from each vertex of the graph to every other vertex. The \textit{strongly connected components} of $G=(V,E,\omega)$ are its maximal strongly connected subgraphs. Moreover, the vertex set $V=\{v_1,\dots,v_n\}$ of $G$ can always be labeled in such a way that $M(G)$ has the triangular block structure

$$M(G)=\left[ \begin{array}{cccc}
M(\mathbb{S}_1(G))      & 0     & \dots  & 0\\
*        & M(\mathbb{S}_2(G))   &        & \vdots\\
\vdots   &       & \ddots & 0\\
*        & \dots & *      & M(\mathbb{S}_m(G))
\end{array}
\right]$$ where $\mathbb{S}_i(G)$ is a strongly connected component of $G$ and $*$ are block matrices with possibly nonzero entries (see \cite{Brualdi91}, \cite{Horn85}, or \cite{Varga09} for more details).

As the strongly connected components of a graph are unique then for $G\in\mathbb{G}^n$ we define
$$\tilde{r}_i(G)=\sum_{j\in N_\ell,j\neq i}|M(\mathbb{S}_\ell(G))_{ij}| \ \text{for} \ 1\leq i\leq n$$
where $i\in N_\ell$ and $N_\ell$ is the set of indices indexing the vertices in $\mathbb{S}_\ell(G)$. That is, $\tilde{r}_i(G)$ is $r_i(G)$ restricted to the strongly connected component containing $v_i$. Furthermore, we let $\tilde{r}_i(\bar{G})=\tilde{r}_i(\bar{G},\lambda)$ where we again consider $\tilde{r}_i(\bar{G},\cdot):\mathbb{C}\rightarrow\mathbb{C}$

If $A\in\mathbb{C}^{n\times n}$ then we write $\tilde{r}_i(G,\lambda)=\tilde{r}_i(A)$ where $A=M(G)$. Moreover, we let $C(A)=C_s(A)\cup C_w(A)$. This allows us to state the following theorem by Varga which, as previously mentioned, is an extension of Brualdi's original theorem \cite{Brualdi82}.

\begin{theorem}{\textbf{(Varga \cite{Varga09})}}\label{theorem5}
Let $A\in\mathbb{C}^{n\times n}$. Then the eigenvalues of $A$ are contained in the set
$$B(A)=\bigcup_{\gamma\in C(A)}\{\lambda\in \mathbb{C}:\prod_{v_i\in\gamma}|\lambda-A_{ii}|\leq\prod_{v_i\in\gamma}\tilde{r}_i(A)\}.$$
\end{theorem}

As with the theorems of Gershgorin and Brauer this result can be extended to matrices in $\mathbb{W}^{n\times n}$.

\begin{theorem}\label{theorem6}
Let $G\in\mathbb{G}$. Then $\sigma(G)$ is contained in the set
\begin{equation}\label{eq3}
\mathcal{BW}_B(G)=\bigcup_{\gamma\in C(\bar{G})}\{\lambda\in\mathbb{C}:\prod_{v_i\in\gamma}|\lambda-M(\bar{G})_{ii}|\leq\prod_{v_i\in\gamma}\tilde{r}_i(\bar{G})\}.
\end{equation}
Also, $\mathcal{BW}_B(G)\subseteq \mathcal{BW}_\mathcal{K}(G)$.
\end{theorem}

We call $\mathcal{BW}_B(G)$ the \textit{Brualdi-type region} of the graph $G$ and the set
$$\mathcal{BW}_B(G)_{\gamma}=\{\lambda\in\mathbb{C}:\prod_{v_i\in\gamma}|\lambda-M(\bar{G})_{ii}|\leq\prod_{v_i\in\gamma}\tilde{r}_i(\bar{G})\}$$ the Brualdi-type region associated with the cycle $\gamma\in C(\bar{G})$.

\begin{proof}
For $G\in\mathbb{G}^n$ let $\bar{G}=\bar{G}(\lambda)$ where for fixed $\alpha\in\mathbb{C}$, $\bar{G}(\alpha)$ is the graph with adjacency matrix $M(\bar{G},\alpha)\in\mathbb{C}^{n\times n}$. Moreover, for any $\gamma=\{v_1,\dots,v_m\}$ in $C(\bar{G})$ and fixed $\alpha\in\mathbb{C}$ let $\gamma(\alpha)$ be the set of vertices $\{v_1,\dots,v_m\}$ in the graph $\bar{G}(\alpha)$.

Using this notation, if $\alpha\in\sigma(G)$ then by lemma \ref{lemma1} and theorem \ref{theorem5} there exists a $\gamma^\prime\in C(\bar{G}(\alpha))$ such that
\begin{equation}\label{eq-3}
\prod_{v_i\in\gamma^\prime}|\alpha-M(\bar{G},\alpha)_{ii}|\leq\prod_{v_i\in\gamma^\prime}\tilde{r}_i(\bar{G},\alpha).
\end{equation}
There are then two possibilities, either $\gamma^\prime\in C(\bar{G})$ or it is not. If $\gamma^\prime\in C(\bar{G})$ then the set of vertices $\gamma^\prime(\alpha)$ is also a cycle in $\bar{G}$ in which case equation (\ref{eq3}) and (\ref{eq-3}) imply $\alpha\in \mathcal{BW}_B(G)$. Suppose then that $\gamma^\prime\notin C(\bar{G})$.

Note that if $\gamma^\prime\in C_s(\bar{G}(\alpha))$ then as $M(\bar{G},\alpha)_{ij}\neq 0$ implies $M(\bar{G},\lambda)_{ij}\neq 0$ for $i\neq j$ then $\gamma^\prime\in C_s(\bar{G})$, which is not possible. Hence, $\gamma^\prime\in C_w(\bar{G}(\alpha))$ or $\gamma^\prime$ must be a loop of some vertex $v_j$ where the graph induced by $\{v_j\}$ in $\bar{G}(\alpha)$ is a strongly connected component of $\bar{G}(\alpha)$. Therefore, equation (\ref{eq-3}) is equivalent to
$|\alpha-M(\bar{G},\alpha)_{jj}|\leq 0$ implying $\alpha=M(\bar{G},\alpha)_{jj}$.

As some cycle $\gamma\in C(\bar{G})$ contains the vertex $v_j$ then $\alpha$ is contained in the set
$$\{\lambda\in\mathbb{C}:\prod_{v_i\in\gamma}|\lambda-M(\bar{G},\lambda)_{ii}|\leq\prod_{v_i\in\gamma}\tilde{r}_i(\bar{G},\lambda)\}$$ implying that $\alpha\in\mathcal{BW}_B(G)$.

To show that $\mathcal{BW}_B(G)\subseteq \mathcal{BW}_\mathcal{K}(G)$ we again follow the proof in \cite{Varga09}. Let $\gamma\in C(\bar{G})$. Supposing that $\gamma\in C_w(\bar{G})$ then $\gamma=\{v_i\}$ for some vertex $v_i$ of $G$ and
$$\mathcal{BW}_B(G)_{\gamma}=\{\lambda\in\mathbb{C}:|\lambda-M(\bar{G},\lambda)_{ii}|=0\}$$
as $v_i$ is the vertex set of some strongly connected component of $\bar{G}$. It follows from (\ref{eq-2}) that $\mathcal{BW}_B(G)_{\gamma}\subseteq\mathcal{BW}_\mathcal{K}(G)_{ij}$ for any $1\leq j \leq n$ where $i\neq j$. In particular, note that if $\tilde{r}_i(\bar{G},\lambda)=0$ then $\lambda\in\mathcal{BW}_\mathcal{K}(G)_{ij}$ for any $1\leq j \leq n$ where $i\neq j$.

If on the other hand, $\gamma\in C_s(\bar{G})$ then for convenience let $\gamma=\{v_1,\dots,v_p\}$ where $p>1$ and note that
\begin{equation}\label{eq20}
\mathcal{BW}_B(G)_{\gamma}=\{\lambda\in\mathbb{C}:\prod_{i=1}^p|\lambda-M(\bar{G},\lambda)_{ii}|\leq\prod_{i=1}^p\tilde{r}_i(\bar{G},\lambda)\}.
\end{equation}

Assuming $0<\tilde{r}_i(\bar{G},\lambda)$ for all $1\leq i\leq p$ then for fixed $\lambda\in \mathcal{BW}_B(G)_{\gamma}$ it follows by raising both sides of the inequality in (\ref{eq20}) to the $(p-1)st$ power that
\begin{equation}\label{eq21}
\prod_{\begin{smallmatrix} 1\leq i,j\leq p\\ i\neq j\end{smallmatrix}}\Big(\frac{|\lambda-M(\bar{G},\lambda)_{ii}||\lambda-M(\bar{G},\lambda)_{jj}|}{\tilde{r}_i(\bar{G},\lambda)\tilde{r}_j(\bar{G},\lambda)}\Big)\leq 1
\end{equation}
As not all the terms of the product in (\ref{eq21}) can exceed unity then for some pair of indices $\ell$ and $k$ where $1\leq \ell,k\leq p$ and $\ell\neq k$ it follows that
\begin{equation}\label{oldeq}
|\lambda-M(\bar{G},\lambda)_{kk}||\lambda-M(\bar{G},\lambda)_{\ell\ell}|\leq \tilde{r}_{k}(\bar{G},\lambda)\tilde{r}_{\ell}(\bar{G},\lambda).
\end{equation}
Using the fact that $\tilde{r}_i(\bar{G},\lambda)\leq r_i(\bar{G},\lambda)$ for all $1\leq i \leq n$ we conclude that $\lambda\in\mathcal{BW}_\mathcal{K}(G)_{k\ell}$ completing the proof.
\end{proof}

The Brualdi-type region for the graph $\mathcal{G}$ with adjacency matrix (\ref{matrix}) is shown in figure 3. We note that
$\mathcal{BW}_B(\mathcal{G})=\mathcal{BW}_\mathcal{K}(\mathcal{G})$ in this particular case.

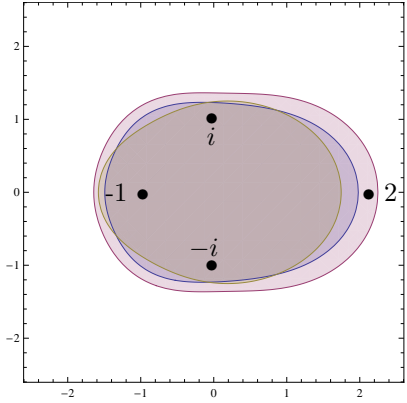
\begin{figure}
\centering
\begin{tabular}{c}
    \begin{overpic}[scale=.5]{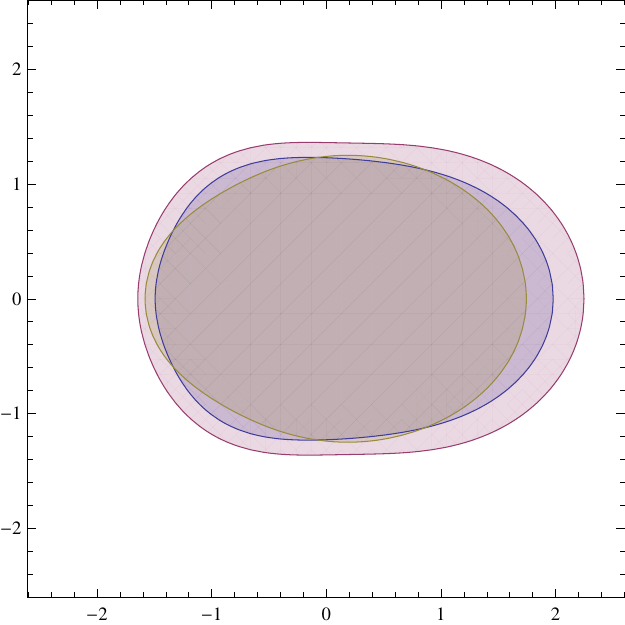}
    \put(25,50){-1 $\bullet$}
    \put(89.5,50){$\bullet$ 2}
    \put(50,32){$\bullet$}
    \put(46,36){$-i$}
    \put(50,69){$\bullet$}
    \put(50.5,64){$i$}
    \end{overpic}
\end{tabular}
\caption{The Brualdi-type region $\mathcal{BW}_B(\mathcal{G})$ for $\mathcal{G}$ in figure 1.}
\end{figure}

We now consider Brualdi's original result which can be stated as follows.

\begin{theorem}{\textbf{(Brualdi \cite{Brualdi82})}}\label{theorem-6}
Let $A\in\mathbb{C}^{n\times n}$ where $C_w(A)=\emptyset$. Then the eigenvalues of $A$ are contained in the set
$$br(A)=\bigcup_{\gamma\in C(A)}\{\lambda\in\mathbb{C}:\prod_{v_i\in \gamma}|\lambda-A_{ii}|\leq \prod_{v_i\in \gamma}r_i(A)\}.$$
\end{theorem}

As with the theorems of Gershgorin, Brauer, and Varga this result generalizes to matrices with entries in $\mathbb{W}$ as follows.

\begin{theorem}\label{theorem-7}
Let $G\in\mathbb{G}$ where $C_w(G)=\emptyset$. Then $\sigma(G)$ is contained in the set
\begin{equation}\label{br}
\mathcal{BW}_{br}(G)=\bigcup_{\gamma\in C(\bar{G})}\{\lambda\in\mathbb{C}:\prod_{v_i\in\gamma}|\lambda-M(\bar{G})_{ii}|\leq\prod_{v_i\in\gamma}r_i(\bar{G})\}.
\end{equation}
Also, $\mathcal{BW}_B(G)\subseteq\mathcal{BW}_{br}(G)\subseteq\mathcal{BW}_{\mathcal{K}}(G)$.
\end{theorem}

\begin{proof}
Note for any graph $G\in\mathbb{G}$ that $\tilde{r}_i(\bar{G})\leq r_i(\bar{G})$ for all $\lambda\in\mathbb{C}$. Hence,
$$\mathcal{BW}_B(G)\subseteq\bigcup_{\gamma\in C(\bar{G})}\{\lambda\in\mathbb{C}:\prod_{v_i\in \gamma}|\lambda-M(\bar{G})_{ii}|\leq \prod_{v_i\in \gamma}r_i(\bar{G})\}.$$
Theorem \ref{theorem6} then implies that $\sigma(G)$ is contained in the set $\mathcal{BW}_{br}(G)$.
Furthermore, if $\tilde{r}_i(G)$ is replaced by $r_i(G)$ in the proof of theorem \ref{theorem6} then in particular (\ref{oldeq}) implies that $\mathcal{BW}_{br}(G)\subseteq\mathcal{BW}_{\mathcal{K}}(G)$, completing the proof.
\end{proof}

We will refer to the region $\mathcal{BW}_{br}(G)$, given in (\ref{br}), as the \textit{original Brualdi-type} region of $G$.

\section{Isospectral Graph Reductions}
Here we present a method developed in \cite{BW09} which allows for the reduction of a graph $G\in\mathbb{G}$ while maintaining the graph's spectrum up to some known set. All results in this section can be found in \cite{BW09} as well as their proofs except for that of theorem \ref{frac} whose proof is contained in this section.

\subsection{Graph Reductions}

In the following if $S\subseteq V$ where $V$ is the vertex set of a graph we let $\bar{S}$ denote the complement of $S$ in $V$. Also if $\{v_1,\dots,v_{m}\}$ is a path in $G\in\mathbb{G}$ let the vertices $v_2,\dots,v_{m-1}$ of $P$ be its \textit{interior} vertices. If $P=\{v_1,\dots,v_m\}$ is a cycle where we fix some $v_i\in P$ then we say $P$ is a cycle from $v_i$ to $v_i$ where $P\setminus\{v_i\}$ are its \textit{interior} vertices.

Recall from section 2 that if we write the graph $G$ as some triple $(V,E,\omega)$ then we are assuming $G\in\mathbb{G}$. With this in mind we give the following definitions.

\begin{definition}\label{def1}
Let $G=(V,E,\omega)$. A nonempty vertex set $S\subseteq V$ is a \textit{structural set} of $G$ if\\
(i) each cycle of $G$, that is not a loop, contains a vertex in $S$; and\\
(ii) $\omega(e_{ii})\neq \lambda$ for each $v_i\in \bar{S}$.
\end{definition}

Part (i) of definition \ref{def1} states that a structural set $S$ of $G$ depends intrinsically on the structure of $G$. Part (ii), however, is the formal assumption that the loops of the vertices in $\bar{S}$, i.e. the complement of $S$, do not have weight equal to $\lambda\in\mathbb{W}[\lambda]$. For $G\in\mathbb{G}$ we let $st(G)$ denote the set of all structural sets of the graph $G$.

\begin{definition}
Suppose $G=(V,E,\omega)$ with structural set $S=\{v_1,\dots,v_m\}$. Let $\mathcal{B}_{ij}(G;S)$ be the set of paths or cycles from $v_i$ to $v_j$ with no interior vertices in $S$. We call a path or cycle $\beta\in\mathcal{B}_{ij}(G;S)$ a \textit{branch} of $G$ with respect to $S$. We let
$$\mathcal{B}_S(G)=\bigcup_{1\leq i,j \leq m} \mathcal{B}_{ij}(G;S)$$
denote the set of all branches of $G$ with respect to $S$.
\end{definition}

If $\beta=v_1,\dots,v_m$ is a branch of $G$ with respect to $S$ and $m>2$ define
\begin{equation}\label{eq0.9}
\mathcal{P}_{\omega}(\beta)=\omega(e_{12})\prod_{i=2}^{m-1}\frac{\omega(e_{i,i+1})}{\lambda-\omega(e_{ii})}.
\end{equation}
For $m=1,2$ let $\mathcal{P}_{\omega}(\beta)=\omega(e_{1m})$. We call $\mathcal{P}_{\omega}(\beta)$ the \textit{branch product} of $\beta$. Note that assumption (ii) in definition \ref{def1} implies that the branch product of any $\beta\in\mathcal{B}_S(G)$ is always defined.

In the procedure we term an \textit{isospectral graph reduction} we replace the branches $\mathcal{B}_{ij}(G;S)$ of a graph with a single edge. The following definition specifies the weights of these edges.

\begin{definition}\label{IR}
Let $G=(V,E,\omega)$ with structural set $S=\{v_1\,\dots,v_m\}$. Define the edge weights
\begin{equation}\label{eq1.0}
\mu(e_{ij})=\begin{cases}
\displaystyle{\sum_{\beta\in\mathcal{B}_{ij}(G;S)}\mathcal{P}_\omega(\beta)} & \text{if} \ \ \ \mathcal{B}_{ij}(G;S)\neq\emptyset\\
\ \ \ \ \ 0 & \text{otherwise}
            \end{cases} \ \ \ \text{for} \ \ \ 1\leq i,j\leq m.
\end{equation}
The graph $\mathcal{R}_S(G)=(S,\mathcal{E},\mu)$ where $e_{ij}\in \mathcal{E}$ if $\mu(e_{ij})\neq 0$
is the \textit{isospectral reduction} of $G$ over $S$.
\end{definition}

\subsection{Sequential Reductions}

As any reduction $\mathcal{R}_S(G)$ of a graph $G\in\mathbb{G}$ is again a graph in $\mathbb{G}$ it is natural to consider sequences of reductions on a graph as well as to what degree a graph can be reduced.

\begin{definition}
For $G=(V,E,\omega)$ suppose the sequence of sets $S_1,\dots,S_m\subseteq V$ are such that $S_1\in st(G)$, $\mathcal{R}_1(G)=\mathcal{R}_{S_1}(G)$ and $$S_{i+1}\in st(\mathcal{R}_i(G)) \ \text{where} \ \mathcal{R}_{S_{i+1}}(\mathcal{R}_i(G))=\mathcal{R}_{i+1}(G), \  1\leq i\leq m-1.$$ If this is the case then we say $S_1,\dots,S_m$ \textit{induces a sequence of reductions} on $G$ with \textit{final vertex} set $S_m$ and we write $\mathcal{R}_i(G)=\mathcal{R}(G;S_1,\dots,S_i)$ for $1\leq i \leq m$.
\end{definition}

\begin{definition}
Let $\mathbb{G}_{\pi}^n\subset\mathbb{G}^n$ be the graphs with weights in the set given by $\{\omega\in\mathbb{W}:\omega=p/q, deg(p)\leq deg(q)\}.$ Furthermore, let $\mathbb{G}_\pi=\bigcup_{n\geq 1}\mathbb{G}^n_\pi$.
\end{definition}

\begin{remark}
It is important to note that any graph $G$ where $M(G)\in\mathbb{C}^{n\times n}$ is a graph in the set $\mathbb{G}^n_\pi$.
\end{remark}

\begin{theorem}\label{theorem-1}
Let $G=(V,E,\omega)$ be in $\mathbb{G}_{\pi}$. Then for any nonempty $\mathcal{V}\subseteq V$ any sequence of reductions on $G$ with final vertex set $\mathcal{V}$ reduces $G$ to the unique graph $\mathcal{R}_{\mathcal{V}}[G]=(\mathcal{V},\mathcal{E},\mu)$. Moreover, at least one such sequence always exists.
\end{theorem}

That is, the final vertex set in a sequence of reductions completely specifies the reduced graph irrespective of the specific sequence of reductions. The notation $\mathcal{R}_\mathcal{V}[G]$ is intended to emphasize the fact that $\mathcal{V}$ need not be a structural set of $G$.

To understand how sequential reductions effect the eigenvalues of a graph (or equivalently matrix) we denote the following. If $G=(V,E,\omega)$ is in $\mathbb{G}_\pi$ where $\mathcal{V}\subseteq V$ let $G|_{\mathcal{V}}$ be the subgraph of $G$ induced over the vertex set $\mathcal{V}$. That is, $$G|_{\mathcal{V}}=(\mathcal{V},\mathcal{E},\mu) \ \ \text{where} \ \ \mathcal{E}=\{e_{ij}\in E: \ v_i,v_j\in \mathcal{V}\} \ \  \text{and} \ \ \mu=\omega|_{\mathcal{E}}.$$

\begin{theorem}\label{frac}
If $G=(V,E,\omega)\in\mathbb{G}_\pi$ where $\mathcal{V}\subseteq V$ is nonempty then
$$\det\big(M(\mathcal{R}_\mathcal{V}[G])-\lambda I\big)=\frac{\det\big(M(G)-\lambda I\big)}{\det\big(M(G|_{\bar{\mathcal{V}}})-\lambda I\big)}.$$
\end{theorem}
\noindent For our purposes, we note that an important interpretation of this theorem is that $\sigma(G)$ and $\sigma(\mathcal{R}_\mathcal{V}[G])$ differ at most by $\sigma(G|_{\bar{\mathcal{V}}})$.

If $A\in\mathbb{C}^{n\times n}$ and $\mathcal{V}\subseteq\{1,\dots,n\}$ is nonempty then let $A|_\mathcal{\bar{V}}$ be the \emph{principle submatrix} of $A$ formed by the rows and columns indexed by $\mathcal{\bar{V}}$. Theorem \ref{frac} then implies
$$\det(A_\mathcal{V}-\lambda I)=\frac{\det(A-\lambda I)}{\det(A|_{\bar{\mathcal{V}}}-\lambda I)}$$
where $A_\mathcal{V}$ is the reduction of $A$ over $\mathcal{V}$ and $\det(A_\mathcal{V}-\lambda I)\in\mathbb{W}$ is the ratio of the characteristic polynomials of $A$ and $A^\mathcal{V}$.

Theorem \ref{frac} also implies the following  useful corollary.

\begin{corollary}\label{theoremnew}
Let $G=(V,E,\omega)$ be a graph in $\mathbb{G}$ and $S\in st(G)$ be a proper subset of $V$. Then
$$\det\big(M(\mathcal{R}_S(G))-\lambda I\big)=\frac{\det\big(M(G)-\lambda I\big)}{\displaystyle{\prod_{v_i\in\bar{S}}(\omega(e_{ii})-\lambda)}}.$$
\end{corollary}

That is, $\sigma(G)$ and $\sigma(\mathcal{R}_S(G))$ differ at most by the set
$$E(G;S)=\{\lambda\in\mathbb{C}:\displaystyle{\prod_{v_i\in\bar{S}}(\omega(e_{ii})-\lambda)}=0\}$$
where this set includes multiplicities. We note that $E(G;S)$ denotes the potential error in estimating $\sigma(G)$ by $\sigma(\mathcal{R}_S(G))$. In particular, if $M(G)\in\mathbb{C}^{n\times n}$ then by reducing $G$ over $S$ we lose any eigenvalues of $M(G)$ which are the weights of the loops $e_{ii}$ for $v_i\in\bar{S}$.

\subsection{Proofs}
In this section we use the following notation. For $G=(V,E,\omega)$ in $\mathbb{G}_{\pi}^n$ and $\mathcal{V}_k=\{v_1,\dots,v_k\}\subset V$ let $M_k=M(\mathcal{R}_{\bar{\mathcal{V}}_k}(G))$ and $M^k=M(G|_{\mathcal{V}_k})$ for $0<k<n$.

For a proof of theorem \ref{frac} we require the following lemma.

\begin{lemma}\label{newlemma20}
For $G\in\mathbb{G}_\pi^n$ where $n>1$, $$\det\big(M(G)-\lambda I\big)=\det\big(M_1-\lambda I\big)\det\big(M^1-\lambda I\big).$$
\end{lemma}

\begin{proof}
If $G\in\mathbb{G}_\pi^{n}$ where $n>1$ then $\bar{\mathcal{V}}_1\in st(G)$. Then lemma \ref{newlemma20} follows from equation (19) of \cite{BW09}.
\end{proof}

A proof of theorem \ref{frac} is the following.

\begin{proof}
For $G=(V,E,\omega)$ in $\mathbb{G}_\pi^n$ let $\mathcal{V}=\mathcal{V}_m$ for some fixed $1\leq m<n$. Denoting $M(G)=M$, lemma \ref{newlemma20} then implies $\det(M-\lambda I)=\det(M_1-\lambda I)\det(M^1-\lambda I)$. Given that the graph corresponding to $M_1$ is in $\mathbb{G}_\pi^{n-1}$ lemma \ref{newlemma20} implies that $$\det\big(M_1-\lambda I\big)=\det\big((M_1)_1-\lambda I\big)\det\big(M_1^1-\lambda I\big).$$ As $(M_1)_1=M_2$ by theorem \ref{theorem-1} then
$\det\big(M_1-\lambda I\big)=\det\big(M_2-\lambda I\big)\det\big(M_1^1-\lambda I\big)$.

By repeated use of both lemma \ref{newlemma20} and theorem \ref{theorem-1} we have
\begin{equation}\label{equation1}
\det(M-\lambda I)=\det(M_m-\lambda I)\prod_{i=1}^m\det\big(M_{i-1}^1-\lambda I\big).
\end{equation} where $M_0=M$.

Denoting $M^m=\tilde{M}$ then, by the same argument, the characteristic equation of the submatrix $M^m$ is given by
\begin{equation}\label{equation2}
\det(\tilde{M}-\lambda I)=\prod_{i=1}^m\det\big(\tilde{M}_{i-1}^1-\lambda I\big).
\end{equation} where $\tilde{M}_0=\tilde{M}$. The claim then is that $\tilde{M}_{i-1}^1=M_{i-1}^1$ for all $1\leq i\leq m$. To verify this we proceed by induction.

First, note that $(M_0)_{jk}=(\tilde{M}_0)_{jk}$ for all $1\leq j,k\leq m$ as $\tilde{M}_0$ is the submatrix of $M_0$ consisting of its first $m$ rows and columns. Therefore, assume that the entries $(M_i)_{jk}=(\tilde{M}_i)_{jk}$ for $1\leq j,k\leq m-i$ and $i<\ell\leq m$. For the case $i=\ell$ it follows from this assumption that
\begin{align}
(M_\ell)_{jk}=(M_{\ell-1})_{j+1,k+1}+\frac{(M_{\ell-1})_{j+1,1}(M_{\ell-1})_{1,k+1}}{\lambda-(M_{\ell-1})_{11}}&=\\(\tilde{M}_{\ell-1})_{j+1,k+1}+\frac{(\tilde{M}_{\ell-1})_{j+1,1}(\tilde{M}_{\ell-1})_{1,k+1}}{\lambda-(\tilde{M}_{\ell-1})_{11}}&=(\tilde{M}_\ell)_{jk}
\end{align} for all $1\leq j,k\leq m-\ell$. Hence, $(M_i)_{jk}=(\tilde{M}_i)_{jk}$ for $1\leq j,k\leq m-i$ and $i\leq m$, verifying the claim that $\tilde{M}_{i-1}^1=M_{i-1}^1$ for all $1\leq i\leq m$.

Given that $\det\big(\tilde{M}_{i-1}^1-\lambda I\big)=\det\big(M_{i-1}^1-\lambda I\big)$ for $1\leq i\leq m$ then equation (\ref{equation1}) together with (\ref{equation2}) imply $\det(M-\lambda I)=\det(M_m-\lambda I)\det(\tilde{M}-\lambda I)$. As $M=M(G)$, $M_m=\mathcal{R}_{\bar{\mathcal{V}}}[G]$, and $\tilde{M}=M(\mathcal{R}^{\mathcal{V}}(G))$ the result follows for the specific set $\mathcal{V}=\mathcal{V}_m$.

To see that this implies the general result of the theorem let $\mathcal{V}$ be any nonempty subset of $V$. By a simple relabeling of the vertices in $\mathcal{V}$ we may write $\mathcal{V}$ as $\mathcal{V}_m$ which completes the proof.
\end{proof}

As the graph $G|_{\bar{S}}$, for any $S\in st(G)$, has only trivial cycles (loops) then each vertex of the graph is its own strongly connected component. Given that the eigenvalues of a graph are the union of the eigenvalues of its strongly connected components
$$\sigma(G|_{\bar{S}})=\{\lambda\in\mathbb{C}:\displaystyle{\prod_{v_i\in\bar{S}}(\omega(e_{ii})-\lambda)}=0\}.$$
This is enough to prove corollary \ref{theoremnew}.

\section{Main Results}
In this section we give the main results of this paper. Specifically, we show that a reduced graph (equivalently reduced matrix) has a smaller Gershgorin and Brauer-type region respectively than the associated unreduced graph. Hence, the eigenvalue estimates given in section 3.1 and 3.2 can be improved via the process of isospectral graph reduction.

However, for both Brualdi and original Brualdi-type regions the situation is more complicated. For certain reductions the Brualdi-type (original Brualdi-type) region of a graph may decrease in size similar to Gershgorin and Brauer-type regions. In other cases the Brualdi-type (original Brualdi-type) region of a graph may increase in size when the graph is reduced. We give an example of both of these possibilities in section 5.3. Following this, we present sufficient conditions under which such estimates improve as the associated graph is reduced (see theorems \ref{impbrualdi} and \ref{bru}).

\subsection{Improving Gershgorin-Type Estimates}
We first consider the effect of reducing a graph on its associated Gershgorin region. Our main result in this direction is the following theorem.
\begin{theorem}{\textbf{(Improved Gershgorin Regions)}}\label{impgersh}
Let $G=(V,E,\omega)$ where $\mathcal{V}$ is any nonempty subset of $V$. If $G\in\mathbb{G}_\pi$ then $\mathcal{BW}_\Gamma(\mathcal{R}_{\mathcal{V}}[G])\subseteq\mathcal{BW}_\Gamma(G)$.
\end{theorem}

Gershgorin's original theorem can be thought of as estimating the spectrum of a graph by considering the paths of length 1 starting at each vertex. Heuristically, one can view graph reductions as allowing for better estimates by considering longer paths in the graph through those vertices that have been removed.

Theorem \ref{impgersh} together with theorem \ref{frac} have the following corollary.

\begin{corollary}\label{cor2}
If $G=(V,E,\omega)$ where $\mathcal{V}$ is a nonempty subset of $V$ then $$\sigma(G)\subseteq\mathcal{BW}_\Gamma(\mathcal{R}_{\mathcal{V}}[G])\cup \sigma(G|_{\bar{\mathcal{V}}}).$$
\end{corollary}

To understand in which situations $\mathcal{BW}_\Gamma(\mathcal{R}_{\mathcal{V}}[G])$ is strictly contained in $\mathcal{BW}_\Gamma(G)$ we consider the following. For $G\in\mathbb{G}^n_\pi$ let
\begin{equation*}
\partial\mathcal{BW}_\Gamma(G)_i=\{\lambda\in\mathbb{C}:|\lambda-M(\bar{G})_{ii}|=r_i(\bar{G},\lambda)\} \ \ \text{for} \ \ 1\leq i\leq n.
\end{equation*}
We note here that the topological boundary of the region $\mathcal{BW}_\Gamma(G)_i$ in the complex plane is contained in the set $\partial\mathcal{BW}_\Gamma(G)_i$ for each $1\leq i\leq n$. This follows from the continuity of $|\lambda-M(\bar{G})_{ii}|-r_i(\bar{G})$ in the variable $\lambda$.
However, if $\lambda\in\partial\mathcal{BW}_\Gamma(G)_i$ it may be the case that $\lambda$ is contained in a neighborhood entirely within $\mathcal{BW}_\Gamma(G)_i$ or $\lambda$ is not on the topological boundary of $\partial\mathcal{BW}_\Gamma(G)_i$. Hence, the topological boundary of $\mathcal{BW}_\Gamma(G)_i$ is contained in $\partial\mathcal{BW}_\Gamma(G)_i$ but this containment may not be strict.

\begin{theorem}\label{theorem8}
Let $G=(V,E,\omega)\in\mathbb{G}_{\pi}^n$. Suppose the subset $$\partial\mathcal{BW}_\Gamma(G)_i\setminus\bigcup_{j=1, j\neq i}^n\mathcal{BW}_\Gamma(G)_j$$
is an infinite set of points. Then $\mathcal{BW}_\Gamma(\mathcal{R}_{\mathcal{V}}[G])\subset\mathcal{BW}_\Gamma(G)$ for any $\mathcal{V}\subset V$ if $v_i\notin\mathcal{V}$.
\end{theorem}

For $G\in\mathbb{G}_{\pi}$ there is typically some region $\mathcal{BW}_\Gamma(G)_i$ whose boundary is not contained in the union of the other $j$th Gershgorin regions. In the nonstandard case this boundary can be a finite set of isolated points but otherwise, removing $v_i$ strictly improves the estimates given by Gershgorin-type regions.

\begin{figure}
\begin{tabular}{ccc}
    \begin{overpic}[scale=.37]{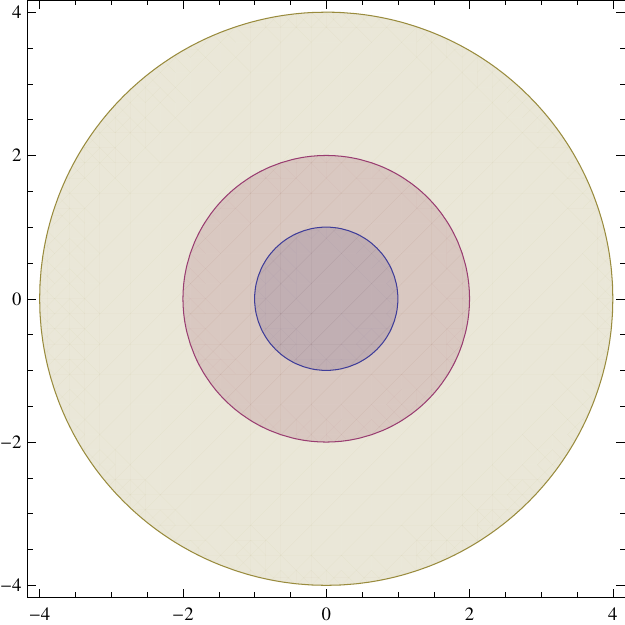}
    \put(22,50){-1 \hspace{0.035in} $\bullet$}
    \put(72.5,50){$\bullet$ 2}
    \put(50,39){$\bullet$}
    \put(45,27){$-i$}
    \put(50,61){$\bullet$}
    \put(52,71){$i$}
    \end{overpic} &
    \begin{overpic}[scale=.37]{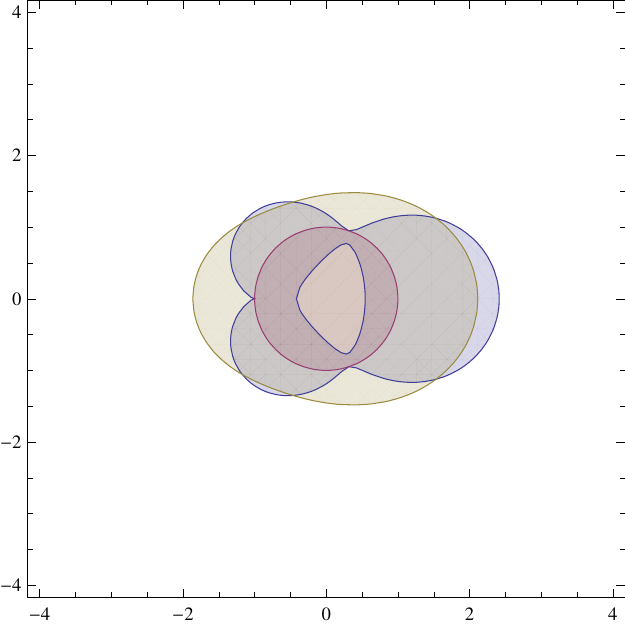}
    \put(24,50){-1 \hspace{0.025in} $\bullet$}
    \put(72.5,50){$\bullet$ 2}
    \put(50,39){$\bullet$}
    \put(45,27){$-i$}
    \put(50,61){$\bullet$}
    \put(52,71){$i$}
    \end{overpic} &
    \begin{overpic}[scale=.37]{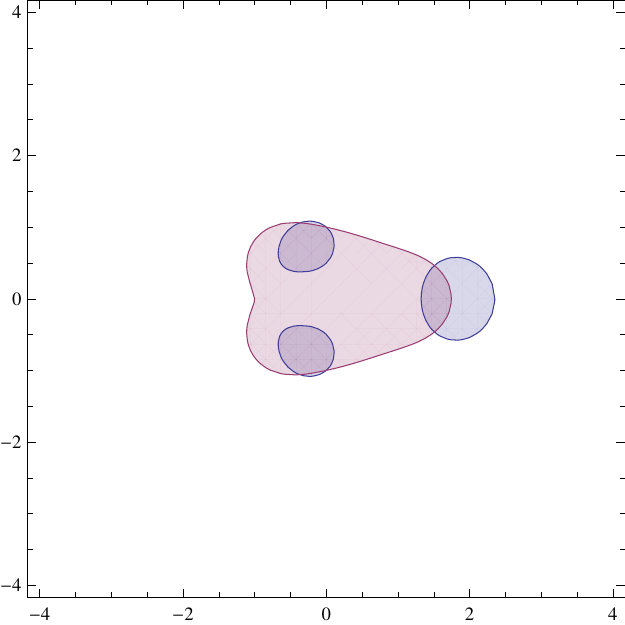}
    \put(27.5,50){-1 $\bullet$}
    \put(72.5,50){$\bullet$ 2}
    \put(50,39){$\bullet$}
    \put(45,31){$-i$}
    \put(50,61){$\bullet$}
    \put(52,67){$i$}
    \end{overpic}
\end{tabular}
\caption{Left: $\mathcal{BW}_\Gamma(\mathcal{G}_0)$. Middle: $\mathcal{BW}_\Gamma(\mathcal{G}_1)$. Right: $\mathcal{BW}_\Gamma(\mathcal{G}_2)$, where in each the spectrum $\sigma(\mathcal{G}_0)=\{-1,-1,-i,i,2\}$ is indicated.}
\end{figure}
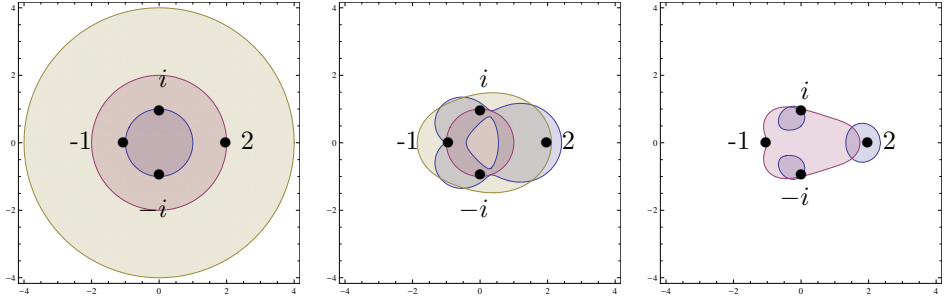

As an example consider the graph $\mathcal{G}_0\in\mathbb{G}_\pi$ with adjacency matrix
$$M(\mathcal{G}_0)=\left[ \begin{array}{ccccc}
0&0&1&0&1\\
0&0&0&1&1\\
0&1&0&0&0\\
1&0&0&0&0\\
1&1&1&1&0
\end{array}
\right].$$
If $\mathcal{G}_1=\mathcal{R}_{\{v_1,v_2,v_3\}}[\mathcal{G}_0]$ and $\mathcal{G}_2=\mathcal{R}_{\{v_1,v_2\}}[\mathcal{G}_1]$ then one computes
\begin{equation}\label{last}
M(\mathcal{G}_1)=\left[ \begin{array}{ccc}
\frac{\lambda+1}{\lambda^2} & \frac{1}{\lambda} & \frac{\lambda+1}{\lambda}\\
\frac{2\lambda+1}{\lambda^2} & \frac{1}{\lambda} & \frac{1}{\lambda}\\
0 & 1 & 0\\
\end{array}
\right] \  \ \text{and} \ \ M(\mathcal{G}_2)=\left[ \begin{array}{cc}
\frac{\lambda+1}{\lambda^2} & \frac{2\lambda+1}{\lambda^2}\\
\frac{2\lambda+1}{\lambda^2} & \frac{\lambda+1}{\lambda^2}\\
\end{array}
\right].
\end{equation}
The Gershgorin regions of $\mathcal{G}_0,$ $\mathcal{G}_1$, and $\mathcal{G}_2$ are shown in figure 4. As $$\partial\mathcal{BW}_\Gamma(\mathcal{G}_0)_5\setminus\bigcup_{j=1}^4\mathcal{BW}_\Gamma(\mathcal{G}_0)_j \ \ \text{and} \ \ \partial\mathcal{BW}_\Gamma(\mathcal{G}_1)_3\setminus\bigcup_{j=1}^2\mathcal{BW}_\Gamma(\mathcal{G}_1)_j$$ consist of curves in $\mathbb{C}$ this, as can be seen in the figure, implies the strict inclusions
$$\mathcal{BW}_\Gamma(\mathcal{G}_2)\subset\mathcal{BW}_\Gamma(\mathcal{G}_1)\subset\mathcal{BW}_\Gamma(\mathcal{G}_0).$$

In addition, if $\mathcal{G}^1=\mathcal{G}_0|_{\{v_4,v_5\}}$ and $\mathcal{G}^2=\mathcal{G}_0|_{\{v_3,v_4,v_5\}}$ then
$$M(\mathcal{G}^1)=\left[ \begin{array}{cc}
0 & 0\\
0 & 1
\end{array}
\right] \  \ \text{and} \ \ M(\mathcal{G}^2)=\left[ \begin{array}{ccc}
0&0&0\\
0&0&0\\
1&1&0
\end{array}
\right].$$
Hence, $\sigma(\mathcal{G}^1)=\sigma(\mathcal{G}^2)=\{0\}$ (not including multiplicities). As $\{0\}$ is contained in both $\mathcal{BW}_\Gamma(\mathcal{G}_1)$ and $\mathcal{BW}_\Gamma(\mathcal{G}_2)$ then both $\mathcal{BW}_\Gamma(\mathcal{G}_1)$ and $\mathcal{BW}_\Gamma(\mathcal{G}_2)$ contain $\sigma(\mathcal{G}_0)$ by corollary \ref{cor2}. (Note $M(\mathcal{G}_1)=M(\mathcal{G})$ where $M(\mathcal{G})$ is previously given by (\ref{matrix}).)

Also, an important implication of theorem \ref{impgersh} is that graph reductions on some $G\in\mathbb{G}_\pi$ can be used to obtain estimates of $\sigma(G)$ with increasing precision depending on how much one is willing to reduce the graph $G$.

With this in mind, suppose $v\in V$ is a vertex of $G\in\mathbb{G}_\pi^n$. Then the graph $\mathcal{R}_{\{v\}}[G]=(\{v\},\mathcal{E},\mu)$ consists of a single vertex $v$ and possibly a loop. We note that this is the furthest extent to which $G$ may be reduced. Moreover, the region $\mathcal{BW}_\Gamma(\mathcal{R}_{\{v\}}[G])=\sigma(\mathcal{R}_{\{v\}}[G])$ is a finite set of points in the complex plane. As $\sigma(G|_{\{V\setminus v\}})$ consists of at most $n-1$ points in the complex plane this can be summarized as follows.

\begin{figure}
\begin{tabular}{ccc}
    \begin{overpic}[scale=.37]{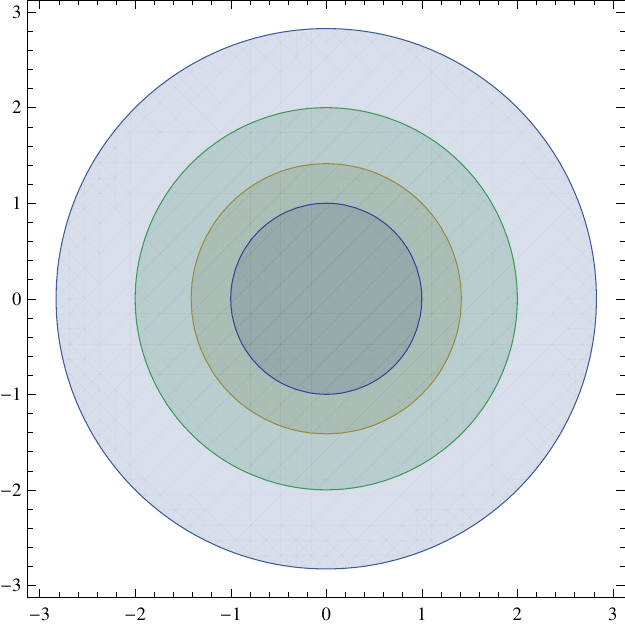}
    \put(22,50){-1 \hspace{0.005in} $\bullet$}
    \put(80.5,50){$\bullet$ 2}
    \put(50,35){$\bullet$}
    \put(46,26){$-i$}
    \put(50,65){$\bullet$}
    \put(53,73){$i$}
    \end{overpic} &
    \begin{overpic}[scale=.37]{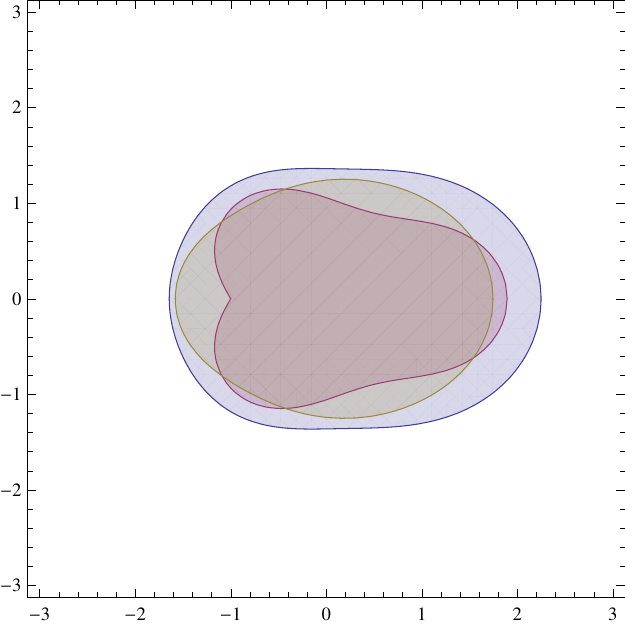}
    \put(20,50){-1 \hspace{0.025in} $\bullet$}
    \put(80.5,50){$\bullet$ 2}
    \put(50,35){$\bullet$}
    \put(46,26){$-i$}
    \put(50,65){$\bullet$}
    \put(53,73){$i$}
    \end{overpic} &
    \begin{overpic}[scale=.37]{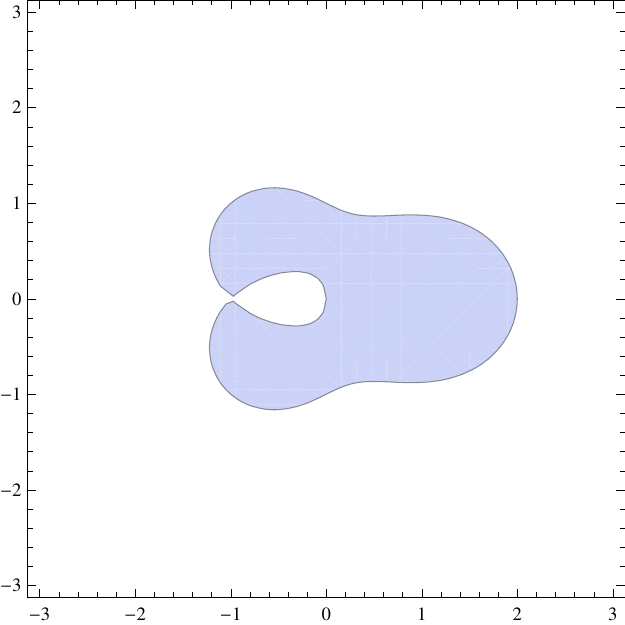}
    \put(24.5,50){-1 $\bullet$}
    \put(80.5,50){$\bullet$ 2}
    \put(50,35){$\bullet$}
    \put(46,26){$-i$}
    \put(50,65){$\bullet$}
    \put(53,73){$i$}
    \end{overpic}
\end{tabular}
\caption{Left: $\mathcal{BW}_\mathcal{K}(\mathcal{G}_0)$. Middle: $\mathcal{BW}_\mathcal{K}(\mathcal{G}_1)$. Right: $\mathcal{BW}_\mathcal{K}(\mathcal{G}_2)$, where in each the spectrum $\sigma(\mathcal{G}_0)=\{-1,-1,-i,i,2\}$ is indicated.}
\end{figure}
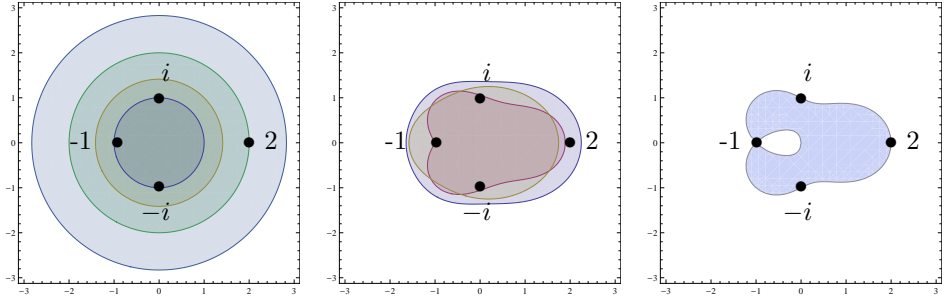

\begin{remark}
If the graph $G=(V,E,\omega)$ is in $\mathbb{G}^n_\pi$ and $v$ is any vertex in $V$ then $\sigma(G)$ is contained in the finite set of points $\sigma(\mathcal{R}_{\{V\setminus v\}}[G])\cup\sigma(\mathcal{R}^{\{V\setminus v\}}(G))$. Furthermore, $\sigma(G)$ and $\sigma(\mathcal{R}_{\{V\setminus v\}}[G])$ differ at most by the set $\sigma(\mathcal{R}^{\{V\setminus v\}}(G))$ which contains less than $n$ points.
\end{remark}

As an example, let $\mathcal{G}_3=\mathcal{R}_{\{v_1\}}[\mathcal{G}_0]$ and $\mathcal{G}^3=\mathcal{G}_0|_{\{v_2,v_3,v_4,v_5\}}$. Then it follows that $\sigma(\mathcal{G}_3)=\{-1,-1,-i,i,2\}$ and $\sigma(\mathcal{G}^3)=\{0,1.3247,-.6623\pm0.5622i\}$. Corollary \ref{cor2} then implies $\sigma(\mathcal{G}_0)\subseteq\{-1,-i,i,2,0,1.3247,-.6623\pm0.5622i\}$. We note that in this particular case $\sigma(\mathcal{G}_0)=\sigma(\mathcal{G}_3)$ or the spectrum of the reduced graph and the original are exactly the same.

\subsection{Improving Brauer-Type Estimates}
We now consider Brauer-type regions for which we give similar results.

\begin{theorem}\label{impbrauer}{\textbf{(Improved Brauer Regions)}}
Let $G=(V,E,\omega)$. If $G\in\mathbb{G}_\pi$ where $\mathcal{V}\subseteq V$ contains at least two vertices, then $\mathcal{BW}_\mathcal{K}(\mathcal{R}_{\mathcal{V}}[G])\subseteq\mathcal{BW}_\mathcal{K}(G)$.
\end{theorem}

Theorem \ref{impbrauer} has the following corollary.

\begin{corollary}
If $G=(V,E,\omega)$ where $\mathcal{V}\subseteq V$ contains at least two vertices then $$\sigma(G)\subseteq\mathcal{BW}_\mathcal{K}(\mathcal{R}_{\mathcal{V}}[G])\cup \sigma(G|_{\bar{\mathcal{V}}}).$$
\end{corollary}

Continuing our example, the Brauer-type regions of $\mathcal{G}_0,\mathcal{G}_1$, and $\mathcal{G}_2$ are shown in figure 5 where by theorem \ref{impbrauer}, $\mathcal{BW}_\mathcal{K}(\mathcal{G}_2)\subseteq\mathcal{BW}_\mathcal{K}(\mathcal{G}_1)\subseteq\mathcal{BW}_\mathcal{K}(\mathcal{G}_0)$. Moreover, theorem \ref{theorem4} implies $\mathcal{BW}_{\mathcal{K}}(\mathcal{G}_i)\subseteq\mathcal{BW}_{\Gamma}(\mathcal{G}_i)$ for $i=0,1,2$.

We note that if a graph is reduced from $n$ to $m$ vertices then there are $\binom{n}{2}-\binom{m}{2}$ less $ij$th Brauer-type regions to calculate. Hence, the number of regions quickly decrease as a graph is reduced.

\subsection{Brualdi-Type Estimates}
Continuing on to Brualdi-type regions we note that in the example we have been considering it happens that we have the inclusions $\mathcal{BW}_B(\mathcal{G}_2)\subseteq\mathcal{BW}_B(\mathcal{G}_1)\subseteq\mathcal{BW}_B(\mathcal{G}_0)$ (see figure 6). However, it is not always the case that reducing a graph will improve its Brualdi-type region.

For example, consider the following graph $\mathcal{H}\in\mathbb{G}_\pi$ given in figure 7. If $\mathcal{H}$ is reduced over the sets $\mathcal{S}=\{v_2,v_3,v_4\}$ and $\mathcal{T}=\{v_1,v_2,v_3\}$ then
$$M(\mathcal{R}_{\mathcal{S}}(\mathcal{H}))=\left[ \begin{array}{ccc}
\frac{1}{\lambda} & \frac{1}{10} & 0 \\
\frac{10}{\lambda} & 0 & 1  \\
0 & 1 & 0 \\
\end{array}
\right] \ \ \text{and} \ \
M(\mathcal{R}_{\mathcal{T}}(\mathcal{H}))=\left[ \begin{array}{ccc}
\frac{1}{\lambda} & \frac{1}{\lambda} & 0\\
1 & 0 & 1 \\
0 & 1 & 0 \\
\end{array}
\right].
$$

In this example we have the strict inclusions (see figure 7) $$\mathcal{BW}_B(\mathcal{R}_{\mathcal{T}}(\mathcal{H}))\subset\mathcal{BW}_B(\mathcal{H})\subset\mathcal{BW}_B(\mathcal{R}_{\mathcal{S}}(\mathcal{H})).$$
In particular, as $\mathcal{BW}_B(\mathcal{H})\subset\mathcal{BW}_B(\mathcal{R}_{\mathcal{S}}(\mathcal{H}))$ then  reducing the graph $\mathcal{H}$ over $\mathcal{S}$ increases the size of its Brualdi-type region. That is, graph reductions do not always improve Brualdi-type estimates.

In order to give a sufficient condition under which a Brualdi-type region shrinks as the graph is reduced we require the following definitions. First, let $G=(V,E,\omega)$ where $V=\{v_1,\dots,v_n\}$ for some $n\geq 1$ and where $G$ has strongly connected components $\mathbb{S}_1(G),\dots,\mathbb{S}_m(G)$. Define $$E^{scc}=\{e\in E: e\in\mathbb{S}_i(G),1\leq i\leq m\}.$$ The cycle $\gamma\in C(G)$ is said to \textit{adjacent} to $v_i\in V$ if $v_i\notin\gamma$ and there is some vertex $v_j\in\gamma$ such that $e_{ji}\in E^{scc}$.

Second, for any $v_i\in V$ we denote $$\mathcal{A}(v_i,G)=\{\gamma\in C(G):\gamma \ \text{is adjacent to} \ v_i\}.$$ Moreover, if $C(v_i,G)=\{\gamma\in C(G):v_i\in\gamma\}$ then let $\mathcal{S}(v_i,G)\subseteq C(v_i,G)$ be the set containing the following cycles.

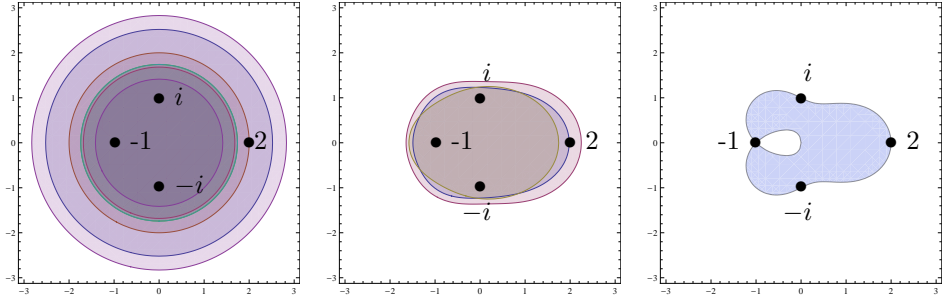
\begin{figure}
\begin{tabular}{ccc}
    \begin{overpic}[scale=.37]{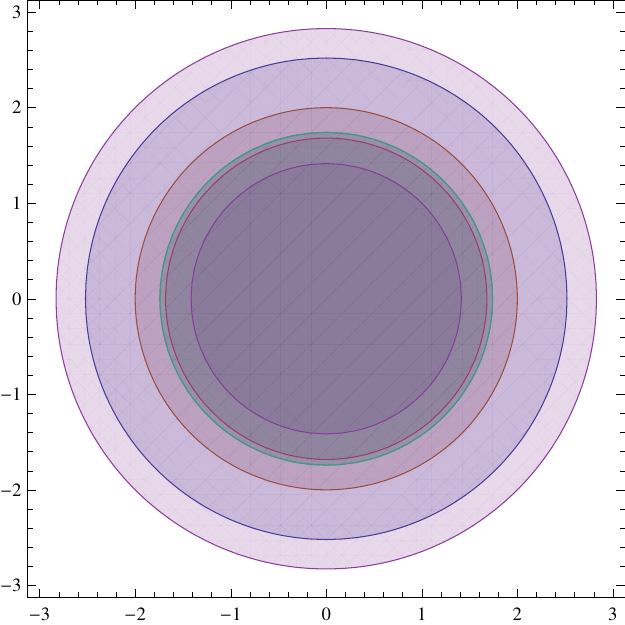}
    \put(35,50){$\bullet$ -1}
    \put(80.5,50){$\bullet$}
    \put(84.5,50){2}
    \put(50,35){$\bullet$ $-i$}
    \put(50,65){$\bullet$ $i$}
    \end{overpic} &
    \begin{overpic}[scale=.37]{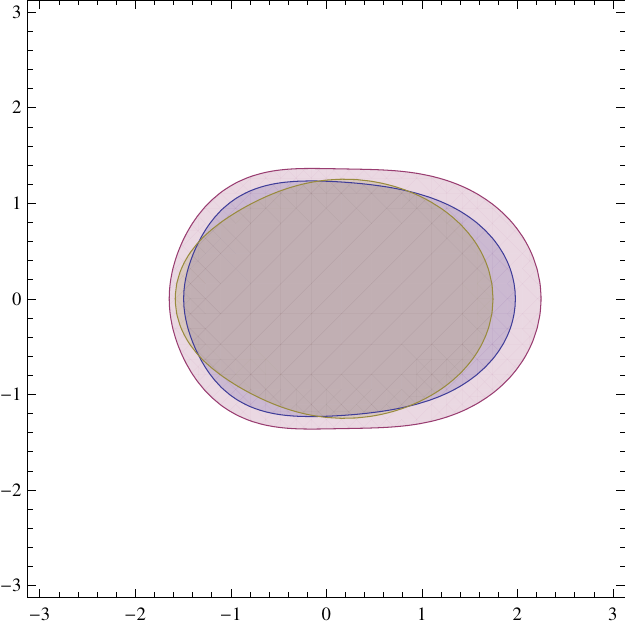}
    \put(35,50){$\bullet$ -1}
    \put(80.5,50){$\bullet$ 2}
    \put(50,35){$\bullet$}
    \put(46,26){$-i$}
    \put(50,65){$\bullet$}
    \put(53,73){$i$}
    \end{overpic} &
    \begin{overpic}[scale=.37]{gershex9.pdf}
    \put(24,50){-1 $\bullet$}
    \put(80.5,50){$\bullet$ 2}
    \put(50,35){$\bullet$}
    \put(46,26){$-i$}
    \put(50,65){$\bullet$}
    \put(53,73){$i$}
    \end{overpic}
\end{tabular}
\caption{Left: $\mathcal{BW}_B(\mathcal{G}_0)$. Middle: $\mathcal{BW}_B(\mathcal{G}_1)$. Right: $\mathcal{BW}_B(\mathcal{G}_2)$, where in each the spectrum $\sigma(\mathcal{G}_0)=\{-1,-1,-i,i,2\}$ is indicated.}
\end{figure}

For $G\in\mathbb{G}_{\pi}^n$ and fixed $1\leq i\leq n$, let $\gamma=\{v_{\alpha_1},\dots,v_{\alpha_m}\}$ be a cycle in $C(v_i,G)$ where $n\geq m\geq 1$ and $v_i=v_{\alpha_1}$. If $m=1$, that is $\gamma=\{v_i\}$, then $\gamma\in\mathcal{S}(v_i,G)$. Otherwise, supposing $1<m\leq n$ relabel the vertices of $G$ such that $v_{\alpha_j}$ is $v_j$ for $1\leq j\leq m$ and denote this relabeled graph by $G_r=(V_r,E_r,\omega_r)$. Then $\gamma\in\mathcal{S}(v_i,G)$ if $e_{j1}\notin E_r$ for $1<j<m$ and $e_{mk}\notin E_r^{scc}$ for $m<k\leq n$.

As it will be needed later, we furthermore define the set $\mathcal{S}_{br}(v_i,G)$ to be the set of cycles in $\mathcal{S}(v_i,G)$ where $\gamma\in\mathcal{S}_{br}(v_i,G)$ if $e_{j1}\notin E_r$ for $1<j<m$ and $e_{mk}\notin E_r$ for $m<k\leq n$.

With this in place we state the following theorem.

\begin{theorem}\label{impbrualdi}{\textbf{(Improved Brualdi Regions)}}
Let $G=(V,E,\omega)$ where $G\in\mathbb{G}_\pi$ and $V$ contains at least two vertices. If $v\in V$ such that both $\mathcal{A}(v,G)=\emptyset$ and $C(v,G)=\mathcal{S}(v,G)$ then $\mathcal{BW}_B(\mathcal{R}_{V\setminus v}(G))\subseteq\mathcal{BW}_B(G)$.
\end{theorem}

That is, if the vertex $v$ is adjacent to no cycle in $C(G)$ and each cycle passing through $v$ is in $\mathcal{S}(v,G)$ then removing this vertex improves the Brualdi-type region of $G$. We note that for the graph $\mathcal{H}$ in figure 7 the set $\mathcal{A}(v_1,\mathcal{H})=\{v_2,v_3\}\neq\emptyset$. Hence, theorem \ref{impbrualdi} does not apply to the reduction of $\mathcal{H}$ over $\mathcal{S}$.

However, the vertex $v_4$ has the property that $\mathcal{A}(v_4,\mathcal{H})=\emptyset$ as well as $\mathcal{S}(v_4,\mathcal{H})=C(v_4,\mathcal{H})$. Therefore, reducing $\mathcal{H}$ over the vertex set $\mathcal{T}=\{v_1,v_2,v_3\}$ improves the Brualdi-type region of this graph which can be seen on the upper right hand side of figure 7.

As an example for why the condition $C(v,G)=\mathcal{S}(v,G)$ is necessary in theorem \ref{impbrualdi} consider the following. Let $\mathcal{J},\mathcal{R}_{S}(\mathcal{J})\in\mathbb{G}$ be the matrices given by
$$M(\mathcal{J})=\left[ \begin{array}{cccc}
0 & 1 & 0 & 0\\
0 & 0 & 1 & 0\\
1 & 0 & 0 & 1\\
1 & 0 & 0 & 0\\
\end{array}
\right], \ \ \text{and} \ \
M(\mathcal{R}_{S}(\mathcal{J}))=\left[ \begin{array}{cccc}
0 & 1 & 0\\
\frac{1}{\lambda} & 0 & 1\\
\frac{1}{\lambda} & 0 & 0\\
\end{array}
\right]$$
where $S=\{v_2,v_3,v_4\}$. In this case $\mathcal{BW}_B(\mathcal{R}_S(\mathcal{J}))\nsubseteq\mathcal{BW}_B(\mathcal{J})$. We note that $\mathcal{A}(v_1,\mathcal{J})=\emptyset$ but $\mathcal{S}(v_1,\mathcal{J})$ consists of the cycle $\{v_1,v_2,v_3\}$ whereas the cycle set  $C(v_1,\mathcal{J})=\{\{v_1,v_2,v_3\},\{v_1,v_2,v_3,v_4\}\}$. That is, $C(v_1,\mathcal{J})\neq \mathcal{S}(v_1,\mathcal{J})$.

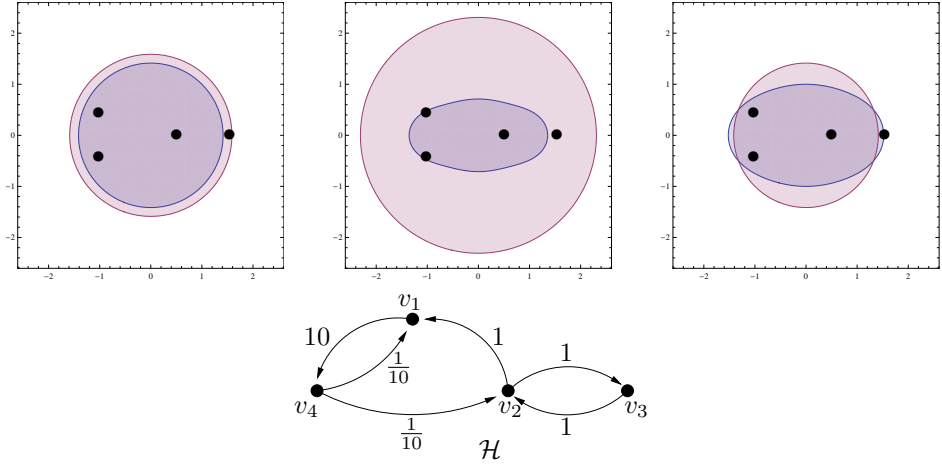
\begin{figure}
\begin{tabular}{ccc}
    \begin{overpic}[scale=.35]{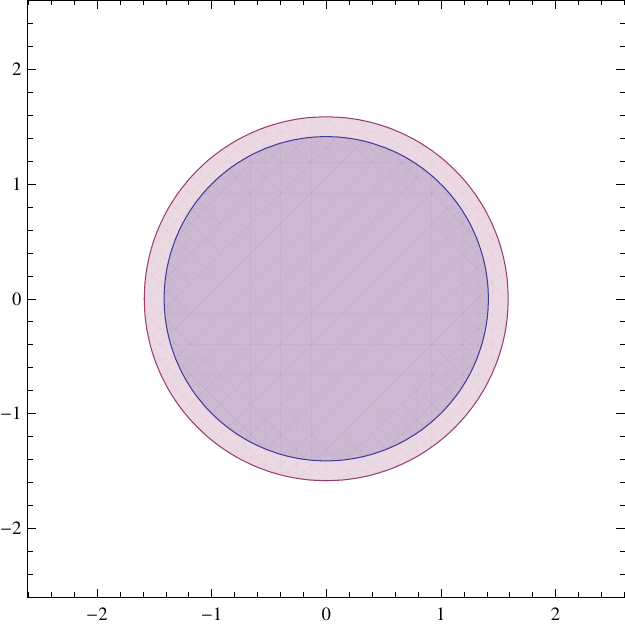}
    \put(59,50){$\bullet$}
    \put(78,50){$\bullet$}
    \put(31,58){$\bullet$}
    \put(31,42){$\bullet$}
    \end{overpic} &
    \begin{overpic}[scale=.35]{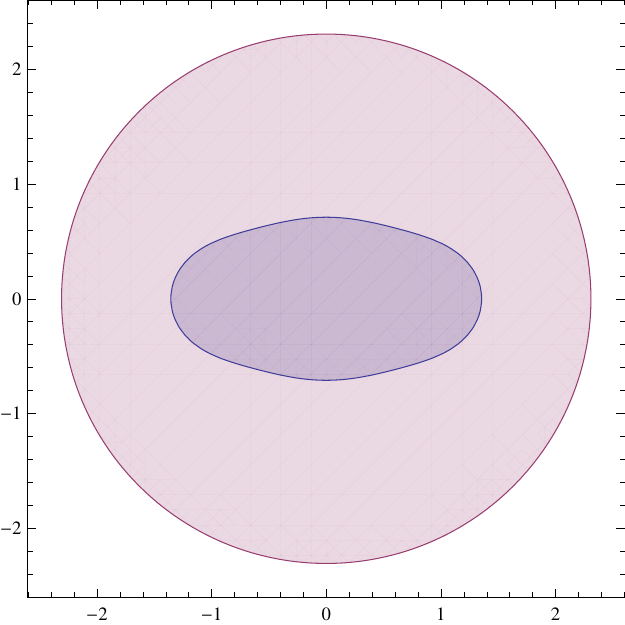}
    \put(59,50){$\bullet$}
    \put(78,50){$\bullet$}
    \put(31,58){$\bullet$}
    \put(31,42){$\bullet$}
    \end{overpic} &
    \begin{overpic}[scale=.35]{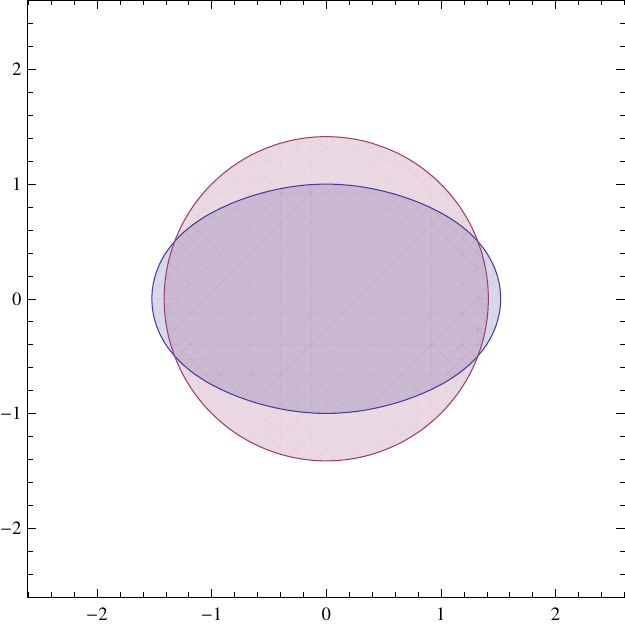}
    \put(59,50){$\bullet$}
    \put(78,50){$\bullet$}
    \put(31,58){$\bullet$}
    \put(31,42){$\bullet$}
    \end{overpic}\\
     &
    \begin{overpic}[scale=.50,angle=270]{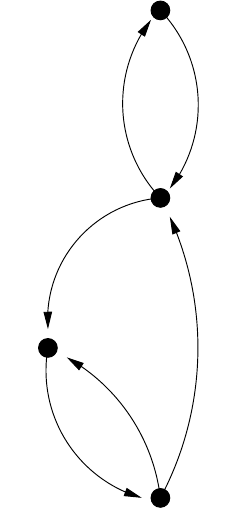}
    \put(52,-4){$\mathcal{H}$}
    \put(-2,31){10}
    \put(27,43){$v_1$}
    \put(27,1.5){$\frac{1}{10}$}
    \put(56,31){1}
    \put(77,3){1}
    \put(58,10){$v_2$}
    \put(77,26){1}
    \put(97,10){$v_3$}
    \put(-5,10){$v_4$}
    \put(23,22){$\frac{1}{10}$}
    \end{overpic} &
\end{tabular}
\caption{Top Left: $\mathcal{BW}_B(\mathcal{H})$. Top Middle: $\mathcal{BW}_B(\mathcal{R}_{\mathcal{S}}(\mathcal{H}))$. Top Right: $\mathcal{BW}_B(\mathcal{R}_{\mathcal{T}}(\mathcal{H}))$ where $\mathcal{S}=\{v_2,v_3,v_4\}$ and $\mathcal{T}=\{v_1,v_2,v_3\}$. $\sigma(\mathcal{H})$ is indicated.}
\end{figure}

Observe that graph reductions can increase, decrease or maintain the number of cycles a graph has in its cycle set. For instance the graph $\mathcal{G}_0$ in our previous example has 12 cycles in its cycle set whereas $\mathcal{G}_1$ has 3 and $\mathcal{G}_2$ has 1 (see figure 6). As an alternate example let $P,$ $\mathcal{R}_{U}(P)\in\mathbb{G}$ with adjacency matrices given by
$$M(P)=\left[ \begin{array}{ccccc}
0 & 1 & 0 & 0 & 0\\
0 & 0 & 0 & 0 & 1\\
0 & 0 & 0 & 1 & 0\\
0 & 0 & 0 & 0 & 1\\
1 & 0 & 1 & 0 & 0
\end{array}
\right], \ \ \text{and} \ \
M(\mathcal{R}_{U}(P))=\left[ \begin{array}{cccc}
0 & 1 & 0 & 0\\
\frac{1}{\lambda} & 0 & \frac{1}{\lambda} & 0\\
0 & 0 & 0 & 1\\
\frac{1}{\lambda} & 0 & \frac{1}{\lambda} & 0\\
\end{array}
\right]$$
where $U=\{v_1,v_2,v_3,v_4\}$. Here $C(P)=\{\{v_1,v_2,v_5\},\{v_3,v_3,v_5\}\}$ whereas\\ $C(\mathcal{R}_{U}(P))=\{\{v_1,v_2\}$, $\{v_3,v_4\},$ $\{v_1,v_2,v_3,v_4\}\}$. That is, reducing $P$ over $U$ increases the number of cycles needed to compute the associated Brualdi-type region from 2 to 3. This is in contrast to Gershgorin and Brauer type regions which always decrease in number as the associated graph is reduced.

In the case of Brualdi's original result (theorem \ref{theorem-7}) we must deal with the following complications. First, for a given graph $G\in\mathbb{G}_\pi$ where $C_w(G)=\emptyset$, it may not be the case that $C_w(\mathcal{R}_{V\setminus v}(G))=\emptyset$. Furthermore, as the edges between strongly connected components play a role in the associated eigenvalue inclusion region (see (\ref{br})) this also complicates whether or not estimates given by the original Brualdi-type region improves as the graph is reduced. However, it is possible to give sufficient conditions under which this is the case.

\begin{theorem}\textbf{(Improved Original Brualdi Regions)}\label{bru}
Let $G=(V,E,\omega)$ be in $\mathbb{G}_\pi$ and $v\in V$. If $\mathcal{A}(v,G)=\emptyset$, $C(v,G)=\mathcal{S}_{br}(v,G)$ and both of the sets $C_w(G)$ and $C_w(\mathcal{R}_{V\setminus v}(G))$ are empty then $\mathcal{BW}_{br}(\mathcal{R}_{V\setminus v}(G))\subseteq\mathcal{BW}_{br}(G)$.
\end{theorem}

\subsection{Proofs}

In order to prove the theorems in section 5.1, 5.2, and 5.3 we will need to evaluate functions at some fixed $\lambda\in\mathbb{C}$. In each case we consider such functions first as elements in $\mathbb{W}$ with common factors removed then evaluated at the value $\lambda$. In fact most of these functions, once common factors are removed, will be polynomials in $\mathbb{C}[\lambda]$.

Moreover, to simplify notation we will use the following. For $G=(V,E,\omega)$ where $G\in\mathbb{G}^n_\pi$ and $n\geq 2$ first note that the vertex set $V\setminus\{v_1\}\in st(G)$. Therefore, let $\mathcal{R}_{V\setminus \{v_1\}}(G)=\mathcal{R}_1$, $L_k(G,\lambda)=L_k$, $L_k(\mathcal{R}_1,\lambda)=L_k^1$, $\lambda-\omega_{kk}=\lambda_{kk}$ and $M(G,\lambda)_{k\ell}=\omega_{k\ell}$. Also, let $\displaystyle{\omega_{k\ell}=p_{k\ell}/q_{k\ell}}$ for $p_{k\ell},q_{k\ell}\in\mathbb{C}[\lambda]$ where we assume $q_{k\ell}=1$ if $\omega_{k\ell}=0$. Lastly, set $R_k(G)=\sum_{\ell=1,\ell\neq k}|\omega_{k\ell}L_k|$.

Before proceeding we state the following lemma.

\begin{lemma}\label{newlemma}
If $G\in\mathbb{G}^n_{\pi}$ for $n\geq 2$ then $q_{11}q_{i1}L_i^1=\big(q_{i1}(q_{11}\lambda-p_{11})\big)^{n-1}L_1L_i$.
\end{lemma}

\begin{proof}
First, note that
$$M(\mathcal{R}_1,\lambda)_{ij}=\frac{p_{i1}p_{1j}q_{ij}q_{11}+q_{i1}q_{1j}p_{ij}(q_{11}\lambda-p_{11})}{q_{i1}q_{1j}q_{ij}(q_{11}\lambda-p_{11})}, \  2\leq i,j\leq n$$ from which $\displaystyle{L_i^1=\prod_{j=2}^nq_{i1}q_{1j}q_{ij}(q_{11}\lambda-p_{11})}$. Therefore,
\begin{equation}\label{star}
\displaystyle{L_i^1=\big(q_{i1}(q_{11}\lambda-p_{11})\big)^{n-1}\prod_{j=2}^nq_{1j}\prod_{j=2}^nq_{ij}}.
\end{equation}
As $\displaystyle{L_k=\prod_{j=1}^n q_{kj}}$ for $1\leq k\leq n$ the result follows by multiplication of $q_{11}q_{i1}$.
\end{proof}

A proof of theorem \ref{impgersh} is the following.

\begin{proof}
Suppose that $\lambda\in\mathcal{BW}_\Gamma(\mathcal{R}_1)_i$ for fixed $\lambda\in\mathbb{C}$ and $2\leq i\leq n$. As each $M(\mathcal{R}_1)_{ij}=\omega_{ij}+\omega_{i1}\omega_{1j}/\lambda_{11}$ for $2\leq j \leq n$ then
$$|(\lambda_{ii}-\frac{\omega_{i1}\omega_{1i}}{\lambda_{11}})L_i^1|\leq\sum_{j=2,j\neq i}^n|(\omega_{ij}+\frac{\omega_{i1}\omega_{1j}}{\lambda_{11}})L_i^1|.$$
Multiplying both sides of this inequality by $|\lambda_{11}q_{11}q_{i1}|$ implies, via lemma \ref{newlemma}, that
$$Q_i(G)|\lambda_{11}L_1\lambda_{ii}L_i-\omega_{i1}\omega_{1i}L_1L_i|\leq Q_i(G)\sum_{j=2,j\neq i}^n|(\omega_{ij}\lambda_{11}+\omega_{i1}\omega_{1j})L_1L_i|$$
where $Q_i(G)=|\big(q_{i1}(q_{11}\lambda-p_{11})\big)|^{n-1}$. If $Q_i(G)\neq 0$ then, by the triangle inequality,
$$|\lambda_{11}L_1\lambda_{ii}L_i|-|\omega_{i1}\omega_{1i}L_1L_i|\leq \sum_{j=2,j\neq i}^n|\lambda_{11}L_1\omega_{ij}L_i|+\sum_{j=2,j\neq i}^n|\omega_{i1}L_i\omega_{1j}L_1|.$$
Therefore,
$$|\lambda_{11}L_1\lambda_{ii}L_i|-\sum_{j=1,j\neq i}^n|\lambda_{11}L_1\omega_{ij}L_i|\leq \sum_{j=2}^n|\omega_{i1}\omega_{1j}L_1L_i|-|\omega_{i1}L_i\lambda_{11}L_1|.$$
By factoring
\begin{equation}\label{neweq}
|\lambda_{11}L_1|\Big(|\lambda_{ii}L_i|-R_i(G)\Big)\leq |\omega_{i1}L_i|\Big(R_1(G)-|\lambda_{11}L_1|\Big).
\end{equation}
If we assume $\lambda\notin\mathcal{BW}_\Gamma(G)_i\cup\mathcal{BW}_\Gamma(G)_1$ then both
$$|\lambda_{ii}L_i|-R_i(G)>0 \ \ \text{and} \ \ R_1(G)-|\lambda_{11}L_1|<0.$$
These inequalities together with (\ref{neweq}) in particular imply that $\lambda_{11}L_1=0$. However, this in turn implies that $\lambda\in\mathcal{BW}_\Gamma(G)_1$, which is not possible.

Hence, $\lambda\in\mathcal{BW}_\Gamma(G)_i\cup\mathcal{BW}_\Gamma(G)_1$ unless $Q_i(G)=0$. Supposing then that this is the case, note that if $\displaystyle{L_{ij}=\prod_{\ell=1,\ell\neq j}^n q_{i\ell}}$ for $1 \leq i,j\leq n$ then
\begin{equation}\label{eq5}
\mathcal{BW}_\Gamma(G)_k=\{\lambda\in\mathbb{C}:|L_{kk}(q_{kk}\lambda-p_{kk})|\leq\sum_{j=1,j\neq k}^n|p_{kj}L_{kj}|\} \ \ \text{for} \ \ 1\leq k \leq n.
\end{equation}

Under the assumption $Q_i(G)=\big(q_{i1}(q_{11}\lambda-p_{11})\big)^{n-1}=0$ note that if $q_{i1}=0$ then $L_{ii}=0$ implying $\lambda\in\mathcal{BW}_\Gamma(G)_i$. If $q_{11}\lambda-p_{11}=0$ then $\lambda\in\mathcal{BW}_\Gamma(G)_1$ again by (\ref{eq5}).

Therefore, $\mathcal{BW}_{\Gamma}(\mathcal{R}_1)_i\subseteq\mathcal{BW}_\Gamma(G)_1\cup\mathcal{BW}_\Gamma(G)_i$ implying $\mathcal{BW}_{\Gamma}(\mathcal{R}_1)\subseteq\mathcal{BW}_\Gamma(G)$. The theorem follows by repeated use of theorem \ref{theorem-1} as it is always possible to sequentially remove single vertices of a graph in order to remove an arbitrary vertex set $\mathcal{\bar{V}}$.
\end{proof}

We now give a proof of theorem \ref{theorem8}.

\begin{proof}
Let $\lambda\in\mathbb{C}$ be fixed such that
\begin{equation}\label{assump1}
\lambda\in\partial\mathcal{BW}_\Gamma(G)_1\setminus\bigcup_{j=2}^n\mathcal{BW}_\Gamma(G)_j.
\end{equation}
Then both
\begin{align}\label{eq8}
|(\lambda_{11})L_1|&=R_1(G); \ \ \text{and}\\
\label{eq9}
|(\lambda_{ii})L_i|&>R_i(G), \ \ \text{for all} \ \ 1<i\leq n.
\end{align}
Supposing $\lambda\in\mathcal{BW}_\Gamma(\mathcal{R}_1)_i$ for some fixed $1<i\leq n$ and that $Q_i(G)\neq 0$ then (\ref{neweq}) holds. Combining (\ref{neweq}) with (\ref{eq8}) it follows that
\begin{equation*}
|\lambda_{11}L_1|\Big(|\lambda_{ii}L_i|-R_i(G)\Big)\leq 0.
\end{equation*}
Moreover, as $|\lambda_{ii}L_i|>R_i(G)$ from equation (\ref{eq9}) then this together with the previous inequality imply that $\lambda_{11}L_1$ must be zero. However, given that $\lambda_{11}L_1$ is a nonzero polynomial then this happens in at most finitely many values of $\lambda\in\mathbb{C}$. Similarly, the polynomial $Q_i(G)=0$ on only a finite set of $\mathbb{C}$, hence the assumption that
$$\partial\mathcal{BW}_\Gamma(G)_1\setminus\bigcup_{j=2}^n\mathcal{BW}_\Gamma(G)_j$$
is an infinite set in the complex plane yields a contradiction to assumption (\ref{assump1}) for infinitely many points in this set. Hence, the result follows in the case that $\{v_1\}=\mathcal{\bar{V}}$. By sequentially removing single vertices of $\bar{\mathcal{V}}$ from the graph $G$ repeated use of theorem \ref{theorem-1} completes the proof.
\end{proof}

Next we give a proof of theorem \ref{impbrauer}.

\begin{proof}
Let $G=(V,E,\omega)$ where $G\in\mathbb{G}^n_\pi$ and $n\geq 3$. The claim is that
\begin{equation}\label{claim1}
\mathcal{BW}_{\mathcal{K}}(\mathcal{R}_1)_{ij}\subseteq\mathcal{BW}_{\mathcal{K}}(G)_{1i}\cup\mathcal{BW}_{\mathcal{K}}(G)_{1j}\cup\mathcal{BW}_{\mathcal{K}}(G)_{ij}
\end{equation}
for any pair $2\leq i,j\leq n$ where $i\neq j$.

To see this let $\lambda\in\mathcal{BW}_{\mathcal{K}}(\mathcal{R}_1)_{ij}$ for fixed $i$ and $j$ from which it follows that
\begin{equation}\label{split1}
\begin{split}
|\big(\lambda_{ii}-\frac{\omega_{i1}\omega_{1i}}{\lambda_{11}}\big)L_i^1||\big(\lambda_{jj}-&\frac{\omega_{j1}\omega_{1j}}{\lambda_{11}}\big)L_j^1|\leq\\ \Big(\sum_{\begin{smallmatrix}\ell=2\\ \ell\neq i\end{smallmatrix}}^n|\big(\omega_{i\ell}+&\frac{\omega_{i1}\omega_{1\ell}}{\lambda_{11}}\big)L_i^1|\Big)\Big(\sum_{\begin{smallmatrix}\ell=2\\ \ell\neq j\end{smallmatrix}}^n|\big(\omega_{j\ell}+\frac{\omega_{j1}\omega_{1\ell}}{\lambda_{11}}\big)L_j^1|\Big)
\end{split}
\end{equation}
Multiplying both sides of (\ref{split1}) by $|\lambda_{11}q_{11}q_{i1}|$ and $|\lambda_{11}q_{11}q_{j1}|$, lemma \ref{newlemma} implies
\begin{align*}
\prod_{k=i,j}Q_k(G)|\lambda_{kk}\lambda_{11}L_1L_k-\omega_{k1}\omega_{1k}L_1L_k|&\leq\\
\prod_{k=i,j}Q_k(G)\Big(\sum_{\begin{smallmatrix}\ell=2\\ \ell\neq k\end{smallmatrix}}^n|&\big(\omega_{k\ell}\lambda_{11}+\omega_{k1}\omega_{1\ell}\big)L_1L_k|\Big).
\end{align*} Assuming for now that $Q_i(G)Q_j(G)\neq 0$ then by the triangle inequality
\begin{equation}\label{align1}
\begin{split}
\prod_{k=i,j}\Big(|\lambda_{11}L_1\lambda_{kk}L_k|-|\omega_{1k}L_1\omega_{k1}L_k|\Big)&\leq\\ \prod_{k=i,j}\Big(\sum_{\begin{smallmatrix}\ell=2\\ \ell\neq k\end{smallmatrix}}^n|\lambda_{11}L_1\omega_{k\ell}L_k|+&\sum_{\begin{smallmatrix}\ell=2\\ \ell\neq k\end{smallmatrix}}^n|\omega_{1\ell}L_1\omega_{k1}L_k|\Big).
\end{split}
\end{equation}
Suppose  $\lambda\notin\mathcal{BW}_{\mathcal{K}}(G)_{1i}\cup\mathcal{BW}_{\mathcal{K}}(G)_{1j}$. Then $|\lambda_{11}L_1||\lambda_{kk}L_k|> R_1(G)R_k(G)$ for $k=i,j$. Moreover, if $|\lambda_{11}L_1|\leq R_1(G)$ then from (\ref{align1})
\begin{equation}\label{align2}
\begin{split}
\prod_{k=i,j}&\Big(R_1(G)R_k(G)-|\omega_{1k}L_1\omega_{k1}L_k|\Big)<\\ &\prod_{k=i,j}\Big(R_1(G)\sum_{\begin{smallmatrix}\ell=2\\ \ell\neq k\end{smallmatrix}}^n|\omega_{k\ell}L_k|+\sum_{\begin{smallmatrix}\ell=2\\ \ell\neq k\end{smallmatrix}}^n|\omega_{k1}L_1\omega_{1\ell}L_k|\Big)
\end{split}
\end{equation}
From the fact that
\begin{equation}\label{eq11}
\begin{split}
R_1(G)R_k(G)-|\omega_{k1}L_1\omega_{1k}L_k|&=\\ R_1(G)&\sum_{\begin{smallmatrix}\ell=2\\ \ell\neq k\end{smallmatrix}}^n|\omega_{k\ell}L_k|+\sum_{\begin{smallmatrix}\ell=2\\ \ell\neq k\end{smallmatrix}}^n|\omega_{k1}L_1\omega_{1\ell}L_k|
\end{split}
\end{equation}
it follows that (\ref{align2}) cannot hold. Therefore, if $\lambda\in\mathcal{BW}_{\mathcal{K}}(\mathcal{R}_1)_{ij}$, $Q_i(G)Q_j(G)\neq 0$, and $\lambda\notin\mathcal{BW}_{\mathcal{K}}(G)_{1i}\cup\mathcal{BW}_{\mathcal{K}}(G)_{1j}$ then $|\lambda_{11}L_1|>R_1(G)$.

Proceeding as before, we assume again that $\lambda\in\mathcal{BW}_{\mathcal{K}}(\mathcal{R}_1)_{ij}$, so in particular (\ref{split1}) holds. Note if $\lambda_{11}=0$ then $\lambda\in\mathcal{BW}_{\mathcal{K}}(G)_{1i}\cup\mathcal{BW}_{\mathcal{K}}(G)_{1j}$ and claim (\ref{claim1}) holds. In what follows we assume then that $\lambda_{11}\neq 0$. Moreover, if $Q_i(G)Q_j(G)\neq0$ then multiplying both side of (\ref{split1}) by $|\lambda_{11}q_{11}q_{i1}|$ and $|\lambda_{ii}L_iq_{11}q_{j1}|$ yields
\begin{equation}\label{align3}
\begin{split}
\Big(|\lambda_{11}L_1\lambda_{ii}L_i|-|\omega_{1i}L_1\omega_{i1}&L_i|\Big)\Big(|\lambda_{ii}L_i\lambda_{jj}L_jL_1|-|\omega_{1j}L_1\omega_{j1}L_j\frac{\lambda_{ii}L_i}{\lambda_{11}}|\Big)\leq\\
\Big(\sum_{\begin{smallmatrix}\ell=2\\ \ell\neq i\end{smallmatrix}}^n|\lambda_{11}L_1\omega_{i\ell}L_i|+\sum_{\begin{smallmatrix}\ell=2\\ \ell\neq i\end{smallmatrix}}^n|&\omega_{1\ell}L_1\omega_{i1}L_i|\Big)\times\\
\Big(&\sum_{\begin{smallmatrix}\ell=2\\ \ell\neq j\end{smallmatrix}}^n|\lambda_{ii}L_i\omega_{j\ell}L_jL_1|+\sum_{\begin{smallmatrix}\ell=2\\ \ell\neq j\end{smallmatrix}}^n|\omega_{1\ell}L_1\omega_{j1}L_j\frac{\lambda_{ii}L_i}{\lambda_{11}}|\Big).
\end{split}
\end{equation}
by use of the triangle inequality.

Supposing that $\lambda\notin\mathcal{BW}_{\mathcal{K}}(G)_{1i}\cup\mathcal{BW}_{\mathcal{K}}(G)_{ij}$ then both $R_1(G)R_i(G)<|\lambda_{11}L_1\lambda_{ii}L_i|$ and $R_i(G)R_j(G)<|\lambda_{ii}L_i\lambda_{jj}L_j|$. This together with (\ref{align3}) implies
\begin{equation*}\label{align4}
\begin{split}
\Big(R_1(G)R_i(G)-|\omega_{1i}L_1\omega_{i1}&L_i|\Big)\Big(R_i(G)R_j(G)L_1-|\omega_{1j}L_1\omega_{j1}L_j\frac{\lambda_{ii}L_i}{\lambda_{11}}|\Big)<\\
\Big(\sum_{\begin{smallmatrix}\ell=2\\ \ell\neq i\end{smallmatrix}}^n|\lambda_{11}L_1\omega_{i\ell}L_i|+\sum_{\begin{smallmatrix}\ell=2\\ \ell\neq i\end{smallmatrix}}^n|&\omega_{1\ell}L_1\omega_{i1}L_i|\Big)\times\\
\Big(|\lambda_{ii}L_iL_1|\big(R_j(G)&-|\omega_{j1}L_j|\big)+|\omega_{j1}L_j\frac{\lambda_{ii}L_i}{\lambda_{11}}|\big(R_1(G)-|\omega_{1j}L_1|\big)\Big).
\end{split}
\end{equation*}
If $|\lambda_{ii}L_i|\leq R_i(G)$ then
\begin{equation}\label{align5}
\begin{split}
\Big(R_1(G)R_i(G)-|\omega_{1i}L_1\omega_{i1}&L_i|\Big)\Big(R_i(G)R_j(G)L_1-|\omega_{1j}L_1\omega_{j1}L_j\frac{\lambda_{ii}L_i}{\lambda_{11}}|\Big)<\\
\Big(\sum_{\begin{smallmatrix}\ell=2\\ \ell\neq i\end{smallmatrix}}^n|\lambda_{11}L_1\omega_{i\ell}L_i|+\sum_{\begin{smallmatrix}\ell=2\\ \ell\neq i\end{smallmatrix}}^n|&\omega_{1\ell}L_1\omega_{i1}L_i|\Big)\cdot\\
\Big(R_i(G)|L_1|\big(R_j(G)&-|\omega_{j1}L_j|\big)+|\omega_{j1}L_j\frac{\lambda_{ii}L_i}{\lambda_{11}}|\big(R_1(G)-|\omega_{1j}L_1|\big)\Big).
\end{split}
\end{equation}
The claim then is that if $\lambda\notin\mathcal{BW}_{\mathcal{K}}(G)_{1i}\cup\mathcal{BW}_{\mathcal{K}}(G)_{1j}$, which implies $|\lambda_{11}L_1|>R_1(G)$ by the above, then the second terms in each product of (\ref{align5}) have the relation
\begin{equation}\label{align6}
\begin{split}
R_i(G)R_j(G)-|\omega_{1j}L_1\omega_{j1}L_j\frac{\lambda_{ii}L_i}{\lambda_{11}}|\geq\\
R_i(G)|L_1|\big(R_j(G)-|\omega_{j1}L_j|\big)+|\omega_{j1}L_j\frac{\lambda_{ii}L_i}{\lambda_{11}}|\big(R_1(G)-|\omega_{1j}L_1|\big).
\end{split}
\end{equation}
To see this note that this is true if and only if
\begin{equation*}
R_i(G)|\omega_{j1}L_jL_1|\geq|\omega_{j1}L_j\lambda_{ii}L_i|\frac{R_1(G)}{|\lambda_{11}|}.
\end{equation*}
As this is true if and only if $|\lambda_{11}L_1|R_i(G)\geq R_1(G)|\lambda_{ii}L_i|$ this verifies that (\ref{align6}) holds since both $R_i(G)\geq|\lambda_{ii}L_i|$ and
$|\lambda_{11}L_1|>R_1(G)$. Therefore, equations (\ref{align5}) and (\ref{align6}) together imply that
\begin{equation}\label{align7}
R_1(G)R_i(G)-|\omega_{1i}L_1\omega_{i1}L_i|<
\sum_{\begin{smallmatrix}\ell=2\\ \ell\neq i\end{smallmatrix}}^n|\lambda_{11}L_1\omega_{i\ell}L_i|+\sum_{\begin{smallmatrix}\ell=2\\ \ell\neq i\end{smallmatrix}}^n|\omega_{1\ell}L_1\omega_{i1}L_i|.
\end{equation}
Rewriting the right-hand side of this inequality in terms of $R_k(G)$ (for $k=1,i$) yields
\begin{equation*}
R_1(G)R_i(G)<|\lambda_{11}L_1|R_i(G)-|\lambda_{11}L_1\omega_{i1}L_i|+|\omega_{i1}L_i|R_1(G).
\end{equation*}
This in turn implies that $R_i(G)\big(R_1(G)-|\lambda_{11}L_1|\big)<|\omega_{i1}L_i|\big(R_1(G)-|\lambda_{11}L_1|\big)$. However, it then follows that $$R_i(G)=\sum^n_{\ell=1,\ell\neq i}|\omega_{i\ell}L_i|<|\omega_{i1}L_i|,$$ which is not possible.

Therefore, if both $Q_i(G)Q_j(G)\neq 0$ and $\lambda\notin\mathcal{BW}_{\mathcal{K}}(G)_{1i}\cup\mathcal{BW}_{\mathcal{K}}(G)_{1j}\cup\mathcal{BW}_{\mathcal{K}}(G)_{ij}$ then $|\lambda_{ii}L_i|>R_i(G)$. Moreover, as this argument is symmetric in the indices $i$ and $j$ then it can be modified to show that if both $Q_i(G)Q_j(G)\neq 0$ and $\lambda\notin\mathcal{BW}_{\mathcal{K}}(G)_{1i}\cup\mathcal{BW}_{\mathcal{K}}(G)_{1j}\cup\mathcal{BW}_{\mathcal{K}}(G)_{ij}$ then $|\lambda_{jj}L_j|>R_j(G)$.

With this in mind, by multiplying (\ref{split1}) by $|q_{11}q_{i1}|$ and $|q_{11}q_{i1}|$ and assuming once again that $Q_i(G)Q_j(G)\neq 0$, then the triangle inequality implies
\begin{equation}\label{align9}
\begin{split}
\prod_{k=i,j}&\Big(|\lambda_{kk}L_k||L_1|-|\frac{\omega_{k1}\omega_{1k}}{\lambda_{11}}L_1L_k|\Big)\leq\\ \prod_{k=i,j}&\Big(\sum_{\begin{smallmatrix}\ell=1\\ \ell\neq k\end{smallmatrix}}^n|\omega_{k\ell}L_k||L_1|-|\omega_{k1}L_kL_1|+\sum_{\ell=2}^n|\frac{\omega_{k1}\omega_{1\ell}}{\lambda_{11}}L_kL_1|-|\frac{\omega_{k1}\omega_{1k}}{\lambda_{11}}L_kL_1|\Big).
\end{split}
\end{equation}
Hence, if $\lambda\notin\mathcal{BW}_{\mathcal{K}}(G)_{1i}\cup\mathcal{BW}_{\mathcal{K}}(G)_{1j}\cup\mathcal{BW}_{\mathcal{K}}(G)_{ij}$ then from the previous calculations $R_k(G)<|\lambda_{kk}L_k|$ for $k=1,i$, and $j$ implying together with (\ref{align9}) that
\begin{equation*}
\begin{split}
\prod_{k=i,j}&\Big(R_k(G)|L_1|-|\frac{\omega_{k1}\omega_{1k}}{\lambda_{11}}L_1L_k|\Big)<\\ &\prod_{k=i,j}\Big(R_k(G)|L_1|-|\omega_{k1}L_kL_1|+|\omega_{k1}L_k|\frac{R_1(G)}{|\lambda_{11}|}-|\frac{\omega_{k1}\omega_{1k}}{\lambda_{11}}L_kL_1|\Big).
\end{split}
\end{equation*}

Hence, for either $k=i$ or $k=j$ it follows that
$$-|\omega_{k1}L_kL_1|+|\omega_{k1}L_k|\frac{R_1(G)}{|\lambda_{11}|}>0.$$
Therefore, $R_1(G)>|\lambda_{11}L_1|$ which is not possible. As this implies that $\lambda\notin\mathcal{BW}_{\mathcal{K}}(G)_{1i}\cup\mathcal{BW}_{\mathcal{K}}(G)_{1j}\cup\mathcal{BW}_{\mathcal{K}}(G)_{ij}$, unless $Q_i(G)Q_j(G)=0$ suppose that this product is in fact equal to zero.

In this case note that by modifying equation (\ref{eq5})
\begin{equation*}
\mathcal{BW}_{\mathcal{K}}(G)_{ij}=\Big\{\lambda\in\mathbb{C}:\prod_{k=i,j}|L_{kk}(q_{kk}\lambda-p_{kk})|\leq\prod_{k=i,j}\Big(\sum_{j=1,j\neq k}^n|p_{kj}L_{kj}|\Big)\Big\}
\end{equation*}
for $1\leq k \leq n$. Hence, if $Q_k(G)=0$ for either $k=i,j$ then by calculations analogous to those given in the proof of theorem \ref{impgersh} it follows that $\lambda\in\mathcal{BW}_{\mathcal{K}}(G)_{ik}$. This verifies the claim given in (\ref{claim1}). Hence, theorem \ref{impbrauer} holds for $\mathcal{V}=V-\{v_1\}$.

As in the previous proofs, theorem \ref{theorem-1} can be invoked to generalize this result to the reduction over the set $\mathcal{V}\subseteq V$.
\end{proof}

In order to prove theorem \ref{impbrualdi} we first give the following lemma.

\begin{lemma}\label{lemma3}
Let $G\in\mathbb{G}^n_\pi$ for $n\geq2$ and suppose both $\mathcal{A}(v_1,G)=\emptyset$ and $C(v_1,G)=\mathcal{S}(v_1,G)$. Moreover, let $\gamma=\{v_1,\dots,v_m\}$ and $\gamma^\prime=\{v_2,\dots,v_m\}$ for $m\geq2$. If $\gamma\in C(G)$ and $\gamma^\prime=\in C(\mathcal{R}_1(G))$ then $\mathcal{BW}_B(\mathcal{R}_1(G))_{\gamma^\prime}\subseteq\mathcal{BW}_B(G)$.
\end{lemma}

\begin{proof}
Suppose first that the hypotheses of the lemma hold. We then make the observation that the edges $e\in E^{scc}$ are not used to calculate to $\mathcal{BW}_B(G)$. Furthermore, any cycle of $G$ is contained in exactly one strongly connected component of this graph. This implies that the Brualdi-type region of the graph is the union of the Brualdi-type regions of its strongly connected components. Therefore, we may without loss in generality assume that $G$ consists of a single strongly connected component.

Suppose that both $\gamma=\{v_1,\dots,v_m\}$ and $\delta=\{v_1,v_m\}$ are cycles in $C(v_1,G)$ for some $1<m\leq n$. Note the fact that $\gamma\in C(v_1,G)$ implies, in particular, that $\gamma^\prime=\{v_2,\dots,v_m\}$ is a cycle in $C(\mathcal{R}_1)$.

From the assumption that $v_1$ has no adjacent cycles it follows that $\omega_{mi}=0$ for $1<i\leq m$ since otherwise $\{v_i,v_{i+1},\dots,v_m\}\in\mathcal{A}(v_1,G)$. Also, as $\gamma\in C(v_1,G)=\mathcal{S}(v_1,G)$ then $\omega_{i1}=0$ for $1<i<m$ as well as $\omega_{mi}=0$ for $m<i\leq n$ as $G$ is assumed to have one strongly connected component. Therefore,
\begin{equation}\label{ineq1}
\mathcal{BW}_B(G)_{\gamma}=\{\lambda\in\mathbb{C}:\prod_{i=1}^m|\lambda_{ii}L_i|\leq |\omega_{m1}L_m|\prod_{i=1}^{m-1}R_i(G)\},
\end{equation}
\begin{equation}\label{ineq2}
\mathcal{BW}_B(G)_\delta=\{\lambda\in\mathbb{C}:|\lambda_{11}L_1||\lambda_{mm}L_m|\leq |\omega_{m1}L_m|R_1(G)\}.
\end{equation}
Suppose then that $\lambda\in\mathcal{BW}_B(\mathcal{R}_1)_{\gamma^\prime}$. Then
\begin{equation}\label{ineq3}
|(\lambda_{mm}-\frac{\omega_{m1}\omega_{1m}}{\lambda_{11}})L_m^1|\prod_{i=2}^{m-1}|\lambda_{ii}L_i|\leq \sum_{i=2}^{m-1}|\frac{\omega_{m1}\omega_{1i}}{\lambda_{11}}L_m^1|\prod_{i=2}^{m-1}R_i(G).
\end{equation}
Here, $L_i^1=L_i$ for $1<i<m$ since for each such $i$ the edge $e_{i1}\notin E$.

Multiplying both sides of (\ref{ineq3}) by $|q_{11}q_{1m}\lambda_{11}|$ along with the triangle inequality implies
\begin{equation}\label{ineq4}\begin{split}
Q_m(G)\Big(|\lambda_{11}L_1\lambda_{mm}L_m|-|\omega_{1m}L_1\omega_{m1}L_m|\Big)\prod_{i=2}^{m-1}|\lambda_{ii}L_i|\leq\\ Q_m(G)\Big(|\omega_{m1}L_m|R_1(G)-|\omega_{1m}L_1\omega_{m1}L_m|\Big)\prod_{i=2}^{m-1}R_i(G).
\end{split}
\end{equation}

Now by use of equation (\ref{eq5}) we have
\begin{equation*}
\mathcal{BW}_{B}(G)_\delta=\{\lambda\in\mathbb{C}:\prod_{k=1,m}|L_{kk}(q_{kk}\lambda-p_{kk})|\leq\prod_{k=1,m}\Big(\sum_{j=1,j\neq k}^n|p_{kj}L_{kj}|\Big)\}.
\end{equation*}
Hence, if $Q_m(G)=0$ then by calculations analogous to those given in the proof of theorem \ref{impgersh} it follows that $\lambda\in\mathcal{BW}_{B}(G)_\delta$. Therefore, assume that $Q_m(G)\neq0$.

Then if $\prod_{i=2}^{m-1}R_i(G)=0$ it follows from (\ref{ineq4}) that either $\prod_{i=2}^{m-1}|\lambda_{ii}L_i|=0$ or that $|\lambda_{11}L_1\lambda_{mm}L_m|-|\omega_{m1}L_1\omega_{1m}L_m|=0$. If the first is the case then $\lambda\in\mathcal{BW}_B(G)_{\gamma}$. If the latter is the case then $\lambda\in\mathcal{BW}_B(G)_{\delta}$ since $|\omega_{1m}L_1|\leq R_1(G)$.

If both $\prod_{i=2}^{m-1}R_i(G)\neq0$ and $|\lambda_{11}L_1\lambda_{mm}L_m|-|\omega_{m1}L_1\omega_{1m}L_m|\neq0$ then (\ref{ineq4}) implies
\begin{equation}\label{ineq5}
\frac{\prod_{i=2}^{m-1}|\lambda_{ii}L_i|}{\prod_{i=2}^{m-1}R_i(G)}\leq \frac{|\omega_{m1}L_m|R_1(G)-|\omega_{1m}L_1\omega_{m1}L_m|}{|\lambda_{11}L_1\lambda_{mm}L_m|-|\omega_{m1}L_1\omega_{1m}L_m|}
\end{equation}
Note that if
$$\frac{|\omega_{m1}L_m|R_1(G)-|\omega_{1m}L_1\omega_{m1}L_m|)}{|\lambda_{11}L_1\lambda_{mm}L_m|-|\omega_{m1}L_1\omega_{1m}L_m|}\leq \frac{|\omega_{m1}L_m|R_1(G)}{|\lambda_{11}L_1\lambda_{mm}L_m|}$$
then it follows from (\ref{ineq5}) together with (\ref{ineq1}) that $\lambda\in\mathcal{BW}_B(G)_{\gamma}$. On the other hand, if
this inequality does not hold then $|\lambda_{11}L_1||\lambda_{mm}L_m|<|\omega_{m1}L_m|R_1(G)$ implying $\lambda\in\mathcal{BW}_B(G)_{\delta}$. Therefore, $\mathcal{BW}_B(\mathcal{R}_1)_{\gamma^\prime}\subseteq\mathcal{BW}_B(G)_{\gamma}\cup\mathcal{BW}_B(G)_{\delta} \subseteq\mathcal{BW}_B(G)$.

Conversely, if $\delta\notin C(G)$ then $\omega_{1m}L_1=0$. Equation (\ref{ineq4}) together with (\ref{ineq1}) then imply that $\mathcal{BW}_B(\mathcal{R}_1)_{\gamma^\prime}\subseteq\mathcal{BW}_B(G)_{\gamma}$. Hence, $\mathcal{BW}_B(G)_{\gamma^\prime}\subseteq\mathcal{BW}_B(G)$.
\end{proof}

We now give a proof of theorem \ref{impbrualdi}.

\begin{proof}
First, as in the previous proof, suppose $G$ consists of a single strongly connected component. Moreover, for the vertex $v_1\in V$ suppose both $\mathcal{A}(v_1,G)=\emptyset$ and $C(v_1,G)=\mathcal{S}(v_1,G)$. Also let $\gamma^\prime=\{v_2,\dots,v_m\}$ be a cycle in $C(\mathcal{R}_1)$ for some $1<m\leq n$.

As $\mathcal{A}(v_1,G)=\emptyset$, if $\gamma^\prime\in C(G)$ then $M(G,\lambda)_{ij}=M(\mathcal{R}_1,\lambda)_{ij}$ for $2\leq i\leq m$ and $1\leq j\leq n$ since $\gamma^\prime$ would otherwise be adjacent to $v_1$. From this it follows that $\mathcal{BW}_B(\mathcal{R}_1)_{\gamma^\prime}=\mathcal{BW}_B(G)_{\gamma^\prime}\subseteq\mathcal{BW}_B(G)$.

On the other hand, if $\gamma^\prime\notin C(G)$ then at least one edge of the form $e_{i-1,i}$ for $3\leq i\leq m$ or $e_{m2}$ is not in $E$. If this is the case then without loss in generality assume for notational simplicity that $e_{m2}\notin E$. Furthermore, let
$$\mathcal{I}=\{i: e_{i-1,i}\notin E, \ 3\leq i\leq m\}\cup\{2\}.$$

We give the set $\mathcal{I}$ the ordering $\mathcal{I}=\{i_1,\dots,i_\ell\}$ such that $i_j<i_k$ if and only if $j<k$. Then for each $1\leq j\leq \ell$ the ordered sets
\begin{equation}\label{set}
\gamma_{j}=\{v_1,v_{i_j},v_{i_j+1},\dots,v_{j_\alpha}\}
\end{equation}
are cycles in $C(v_1,G)$ where $j_{\alpha}=i_{j+1}-1$ and $\ell_{\alpha}=m$. Moreover, by removing the vertex $v_1$ from $G$ it follows from (\ref{set}) that each of the ordered sets
$$\gamma_{j}^\prime=\{v_{i_j},v_{i_j+1},\dots,v_{j_\alpha}\}$$
are cycles in $C(\mathcal{R}_1)$. As both $\mathcal{A}(v_1,G)=\emptyset$ and $C(v_1,G)=\mathcal{S}(v_1,G)$, lemma \ref{lemma3} therefore implies that
$$\bigcup^\ell_{j=1}\mathcal{BW}_B(\mathcal{R}_1)_{\gamma_j^\prime}\subseteq\mathcal{BW}_B(G).$$

The claim then is that the region
\begin{equation}\label{ineq6}
\mathcal{BW}_B(\mathcal{R}_1)_{\gamma^\prime}\subseteq\bigcup_{j=1}^{\ell}\mathcal{BW}_B(\mathcal{R}_1)_{\gamma_j^\prime}.
\end{equation}

To see this, let $\displaystyle{\lambda_{ii}^1=(\lambda-\omega_{ii}-\frac{\omega_{i1}\omega_{1i}}{\lambda_{11}})L_i^1}$ and $\displaystyle{R^1_i=\sum_{j=2,j\neq i}^n|M(\bar{\mathcal{R}}_1,\lambda)_{ij}|}$. Then
\begin{equation}\label{ref1}
\mathcal{BW}_B(\mathcal{R}_1)_{\gamma^\prime}=\{\lambda\in\mathbb{C}:\prod_{i=2}^m|\lambda_{ii}^1|\leq\prod_{i=2}^m R_i^1\} \ \ \text{and}
\end{equation}
\begin{equation}\label{ref2}
\mathcal{BW}_B(\mathcal{R}_1)_{\gamma_j^\prime}=\{\lambda\in\mathbb{C}:\prod_{i\in\gamma_j}^m|\lambda_{ii}^1|\leq\prod_{i\in\gamma_j}^m R_i^1\} \ \ \text{for} \ \ 1\leq j\leq \ell.
\end{equation}

As the vertex set $\gamma^\prime$ is the disjoint union of the vertex sets of the cycles $\gamma^\prime_j$ then
the assumption that $\lambda\notin\mathcal{BW}_B(\mathcal{R}_1)_{\gamma_j^\prime}$ for each $1\leq j\leq \ell$ implies $\lambda\notin\mathcal{BW}_B(\mathcal{R}_1)_{\gamma^\prime}$ by comparing the product of (\ref{ref2}) over all $1\leq j\leq \ell$ to (\ref{ref1}). This verifies the claim given in (\ref{ineq6}), which implies that $\mathcal{BW}_B(\mathcal{R}_1)_{\gamma^\prime}\subseteq\mathcal{BW}_B(G)$.

As $\gamma^\prime$ was an arbitrary cycle in $C(\mathcal{R}_1)$ then it follows that $\mathcal{BW}_B(\mathcal{R}_1)\subseteq\mathcal{BW}_B(G)$. This completes the proof.
\end{proof}

A proof of theorem \ref{bru} is the following.

\begin{proof}
If the conditions given in the theorem hold for $v=v_1$ then both $\mathcal{BW}_{br}(G)$ and $\mathcal{BW}_{br}(\mathcal{R}_1)$ exist since it is assumed that $C_w(G)=\emptyset$ and $C_w(\mathcal{R}_1)=\emptyset$. Moreover, if $\mathcal{S}(v_1,G)$ is replaced by $\mathcal{S}_{br}(v_1,G)$ and $\mathcal{BW}_B(\cdot)$ by $\mathcal{BW}_{br}(\cdot)$ then the conclusions of lemma \ref{lemma3} hold by the same proof following the lemma with the exception that $G$ is not assumed to have a single strongly connected component. As the same holds for the proof of theorem \ref{impbrualdi} the result follows.
\end{proof}

\section{Some Applications} In this section we discuss some natural applications of using graph reductions to improve estimates of the spectra of certain graphs. Our first application deals with estimating the spectra of the Laplacian matrix of a given graph. Following this we give a method for estimating the spectral radius of a matrix using graph reductions. Last, we use the results of theorem \ref{theorem8} as well as some structural knowledge of a graph to identify particularly useful structural sets.

\subsection{Laplacian Matrices}
It is possible to reduce not only the graph $G$ but also the graphs associated with both the combinatorial Laplacian matrix and the normalized Laplacian matrix of $G$. Such matrices are typically defined for undirected graphs without loops or weights but this definition can be extended to graphs in $\mathbb{G}$ (see remark 3 below). However, here we give the standard definitions as these are of interest in their own right (see \cite{Chung97,Chung06}).

\begin{figure}
\begin{tabular}{cc}
    \begin{overpic}[scale=.5]{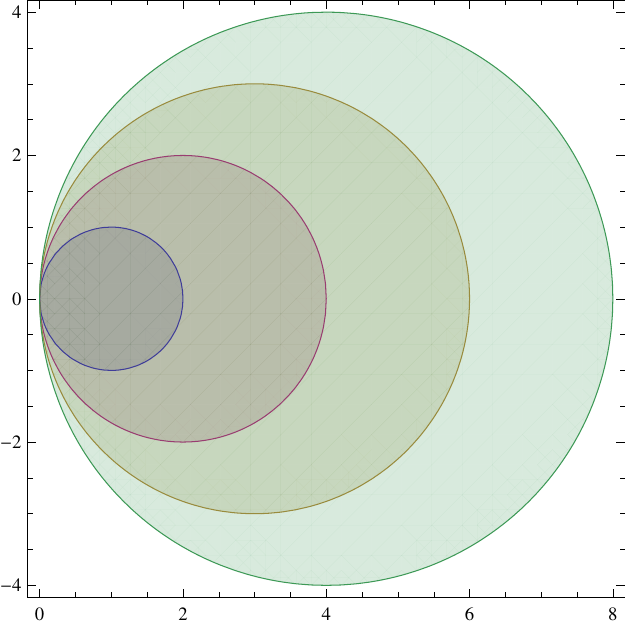}
    \put(5,50){$\bullet$ 0}
    \put(16.5,50){$\bullet$ 1}
    \put(28,50){$\bullet$ 2}
    \put(50.5,50){$\bullet$ 4}
    \put(62,50){$\bullet$ 5}
    \end{overpic} &
    \begin{overpic}[scale=.5]{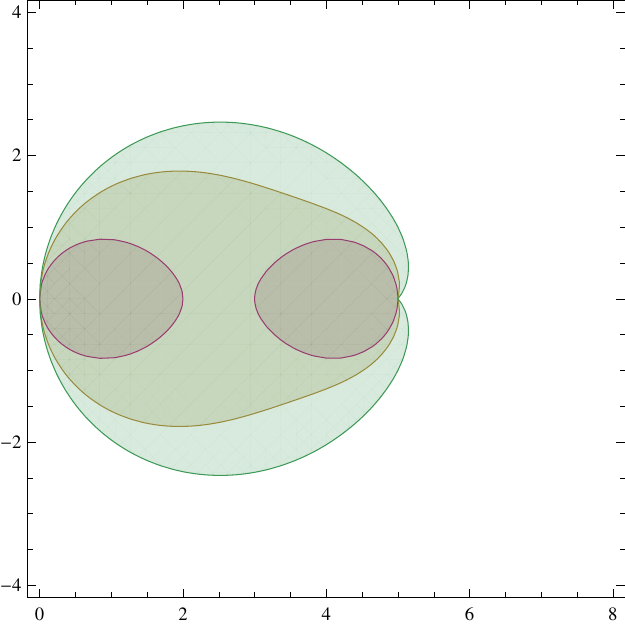}
    \put(5,50){$\bullet$ 0}
    \put(16.5,50){$\bullet$ 1}
    \put(28,50){$\bullet$ 2}
    \put(50.5,50){$\bullet$ 4}
    \put(62,50){$\bullet$ 5}
    \end{overpic}
\end{tabular}
\caption{Left: $\mathcal{BW}_\Gamma\big(L(H)\big)$. Right: $\mathcal{BW}_\Gamma\big(\mathcal{R}_S(L(H))\big)$, where in each the spectrum $\sigma\big(L(H)\big)=\{0,1,2,4,5\}$ is indicated.}
\end{figure}
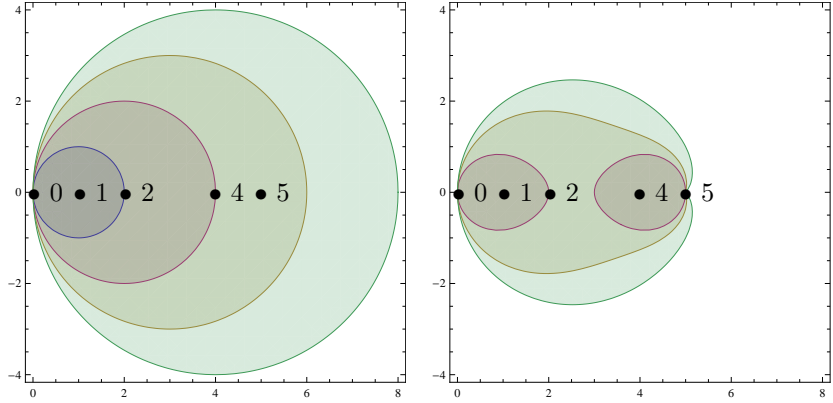

Let $G=(V,E)$ be an unweighted undirected graph without loops, i.e. a \textit{simple graph}. If $G$ has vertex set $V=\{v_1,\dots,v_n\}$ and $d(v_i)$ is the degree of vertex $v_i$ then its \textit{combinatorial Laplacian matrix} $M_L(G)$ of $G$ is given by
$$M_L(G)_{ij}=\begin{cases}
d(v_i) & if \ \ i=j\\
-1 & if \ \ i\neq j \ \text{and} \ v_i \ \text{is adjacent to} \ v_j\\
0 & \text{otherwise}
\end{cases}$$ On the other hand the \textit{normalized Laplacian matrix} $M_\mathcal{L}(G)$ of $G$ is defined as
$$M_\mathcal{L}(G)_{ij}=\begin{cases}
1 & if \ \ i=j \ \text{and} \ d(v_j)\neq 0\\
\frac{-1}{\sqrt{d(v_i)d(v_j)}} & if \ \ v_i \ \text{is adjacent to} \ v_j\\
0 & \text{otherwise}
\end{cases}$$

The interest in the eigenvalues of $M_L(G)$ is that $\sigma(M_L(G))$ gives structural information about $G$ (see \cite{Chung97}). On the other hand knowing $\sigma(M_\mathcal{L}(G))$ is useful in determining the behavior of algorithms on the graph $G$ among other things (see \cite{Chung06}).

Let $L(G)$ be the graph with adjacency matrix $M_L(G)$ and similarly let $\mathcal{L}(G)$ be the graph with adjacency matrix $M_\mathcal{L}(G)$. Since both $L(G),\mathcal{L}(G)\in\mathbb{G}_\pi$ either may be reduced over any subset of their respective vertex sets.

For example if $H\in\mathbb{G}_\pi$ is the simple graph with adjacency matrix
$$M(H)=\left[ \begin{array}{ccccc}
0&0&0&0&1\\
0&0&0&1&1\\
0&0&0&1&1\\
0&1&1&0&1\\
1&1&1&1&0
\end{array}
\right]$$
then the graph $L(H)$, has the structural set $S=\{v_1,v_2,v_3,v_4\}$. Reducing over this set yields $\mathcal{R}_S\big(L(H)\big)$ where
$$M\big(\mathcal{R}_S(L(H))\big)=\left[ \begin{array}{cccc}
\frac{\lambda-3}{\lambda-4}&\frac{1}{\lambda-4}&\frac{1}{\lambda-4}&\frac{1}{\lambda-4}\\
\frac{1}{\lambda-4}&\frac{2\lambda-7}{\lambda-4}&\frac{1}{\lambda-4}&\frac{-\lambda+5}{\lambda-4}\\
\frac{1}{\lambda-4}&\frac{1}{\lambda-4}&\frac{2\lambda-7}{\lambda-4}&\frac{-\lambda+5}{\lambda-4}\\
\frac{1}{\lambda-4}&\frac{-\lambda+5}{\lambda-4}&\frac{-\lambda+5}{\lambda-4}&\frac{3\lambda-11}{\lambda-4}
\end{array}
\right].$$
Figure 8 shows the Gershgorin regions for $L(H)$ as well as $\mathcal{R}_S(L(H))$.

Note that the adjacency matrix of $H$ is symmetric so its eigenvalues must be real numbers. With this in mind we note that the Gershgorin-type region associated with simple graphs and their reductions can be reduced to intervals of the real number line.

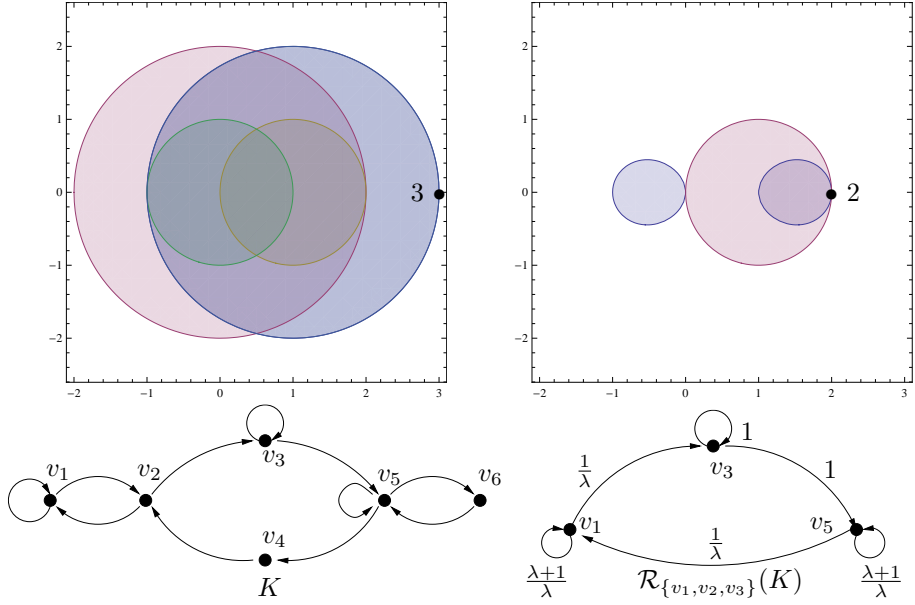
\begin{figure}
\begin{tabular}{cc}
    \begin{overpic}[scale=.5]{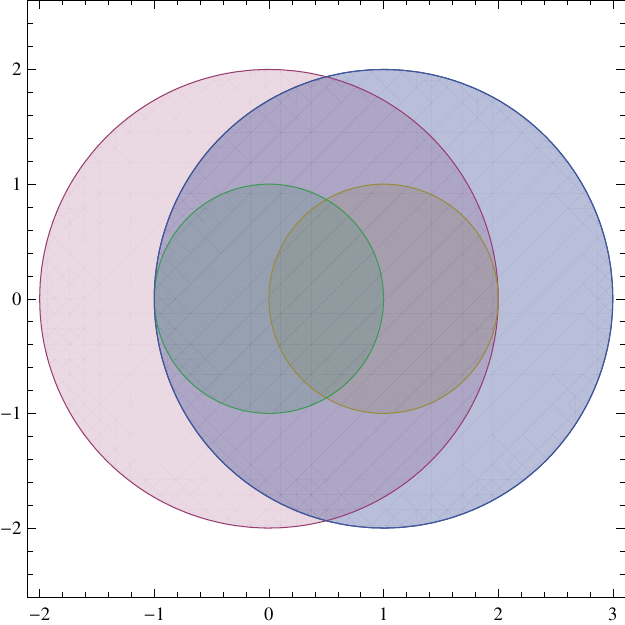}
    \put(91,50){3 $\bullet$}
    \end{overpic} &
    \begin{overpic}[scale=.5]{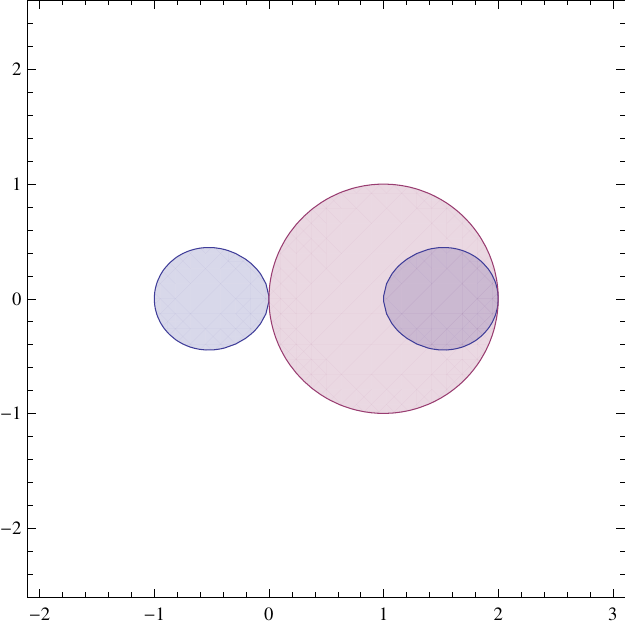}
    \put(78,50){$\bullet$ 2}
    \end{overpic}\\
    \begin{overpic}[scale=.5,angle=270]{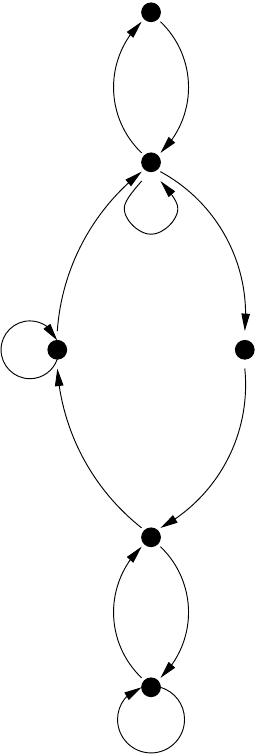}
    \put(52,-6){$K$}
    \put(8,19){$v_1$}
    \put(27,19){$v_2$}
    \put(53,22){$v_3$}
    \put(53,5){$v_4$}
    \put(77,18){$v_5$}
    \put(98,18){$v_6$}
    \end{overpic} &
    \begin{overpic}[scale=.5,angle=270]{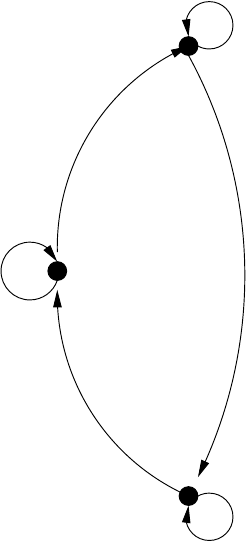}
    \put(11,11){$v_1$}
    \put(78,11){$v_5$}
    \put(49,27){$v_3$}
    \put(28,-6){$\mathcal{R}_{\{v_1,v_2,v_3\}}(K)$}
    \put(82,26){$1$}
    \put(58,37){$1$}
    \put(48,5){$\frac{1}{\lambda}$}
    \put(10,26){$\frac{1}{\lambda}$}
    \put(-5,-6){$\frac{\lambda+1}{\lambda}$}
    \put(92,-6){$\frac{\lambda+1}{\lambda}$}
    \end{overpic}
\end{tabular}
\caption{Top Left: $\mathcal{BW}_\Gamma(K)$ from which  $\rho(K)\leq 3$. Top Right: $\mathcal{BW}_\Gamma(\mathcal{R}_{\{v_1,v_2,v_3\}}(K))$ from which $\rho(K)\leq 2$.}
\end{figure}

\begin{remark}
It is possible to generalize $M_L(G)$ to any $G\in\mathbb{G}$ if $G$ has no loops and $n$ vertices by setting $M_L(G)_{ij}=-M(G)_{ij}$ for $i\neq j$ and $M_L(G)_{ii}=\sum_{j=1,j\neq i}^n M(G)_{ij}$. This generalization is consistent with what is done for weighted digraphs in \cite{Wu05} for example.
\end{remark}

\subsection{Estimating the Spectral Radius of a Matrix}
For $G\in\mathbb{G}_{\pi}$ the spectral radius of $G$, denoted $\rho(G)$, is the maximum among the absolute values of the elements in $\sigma(G)$ i.e. $$\rho(G)=\max_{\lambda\in\sigma(G)}|\lambda|.$$

For many graphs $G\in\mathbb{G}_\pi$ it is possible to find some structural set $S\in st(G)$ such that each vertex of $\bar{S}$ has no loop. By corollary \ref{theoremnew}, if $S$ is such a set then $\sigma(G)$ and $\sigma(\mathcal{R}_S(G))$ differ at most by $E(G;S)=\{0\}$ implying that $\rho(G)=\rho(\mathcal{R}_S(G))$.

For example, in the graph $K$ shown in figure 9 the vertices $v_2,v_4,v_6$ are the vertices of $K$ without loops. As $\{v_1,v_3,v_5\}\in st(K)$ it follows that $\rho(K)=\rho(\mathcal{R}_{\{v_1,v_3,v_5\}}(K))$.

By employing the region $\mathcal{BW}_\Gamma(K)$ we can estimate  $\rho(K)\leq 3$. However, using $\mathcal{BW}_\Gamma(\mathcal{R}_{\{v_1,v_3,v_5\}}(K))$ our estimate improves to $\rho(K)\leq 2$ (see the top left and right of figure 9).

It should be noted that for a given graph there is often no unique set of vertices without loops which is simultaneously a structural set. Therefore, there may be many ways to reduce a graph such that at each step only vertices without loops are removed ensuring, as above, that the spectral radius is maintained.

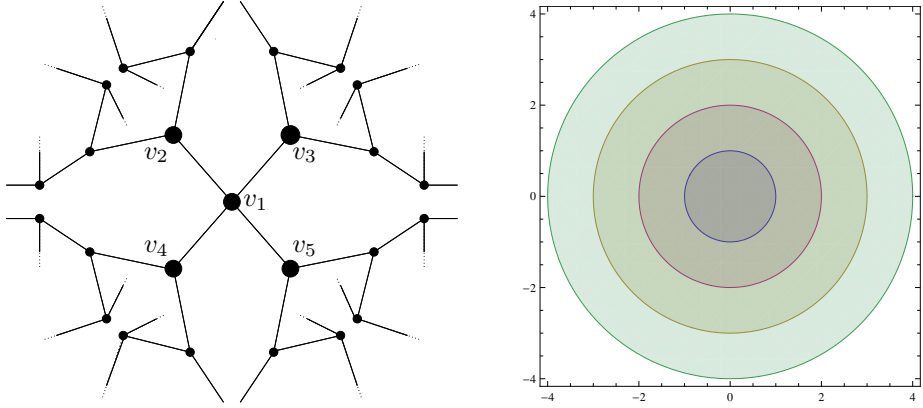
\begin{figure}
\begin{tabular}{cc}
    \begin{overpic}[scale=.35,angle=270]{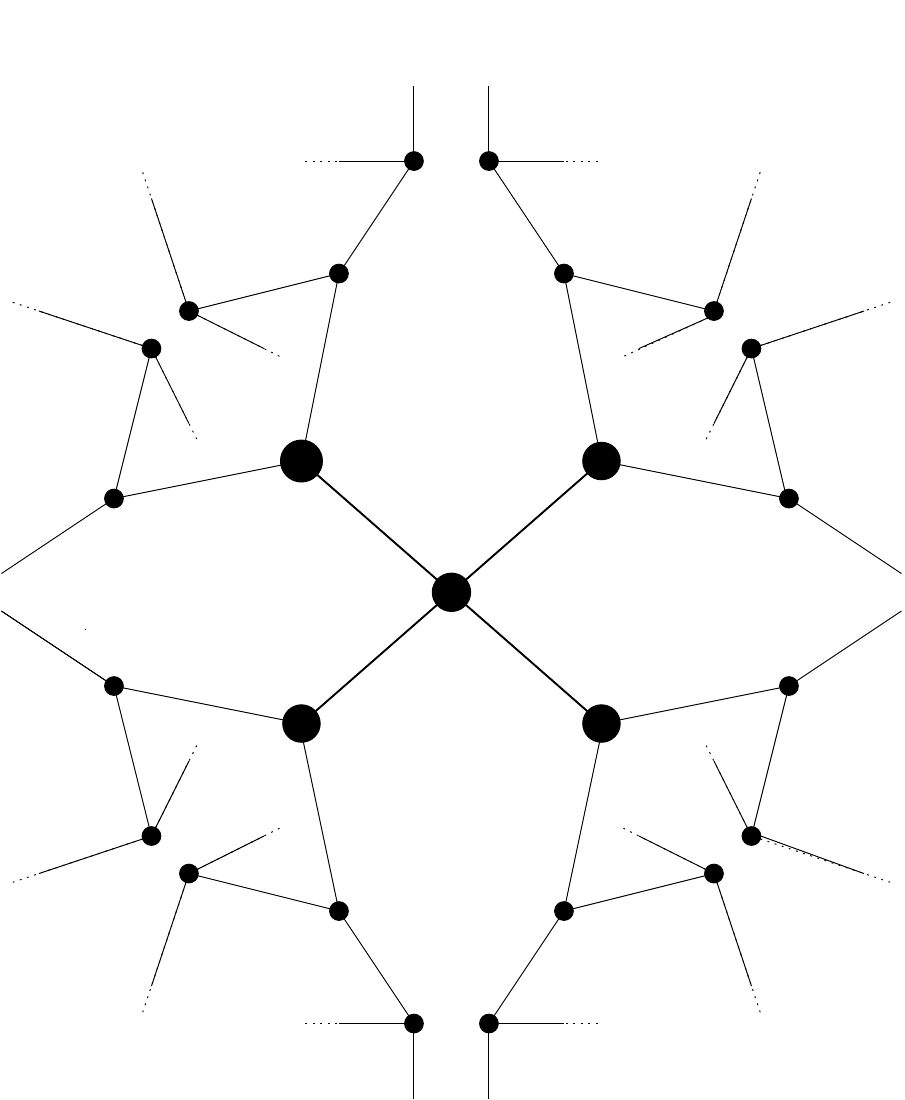}
    \put(58.5,31){$v_5$}
    \put(28.25,50){$v_2$}
    \put(58.5,50){$v_3$}
    \put(28.25,31){$v_4$}
    \put(48.5,40.25){$v_1$}
    \end{overpic} &
    \begin{overpic}[scale=.5]{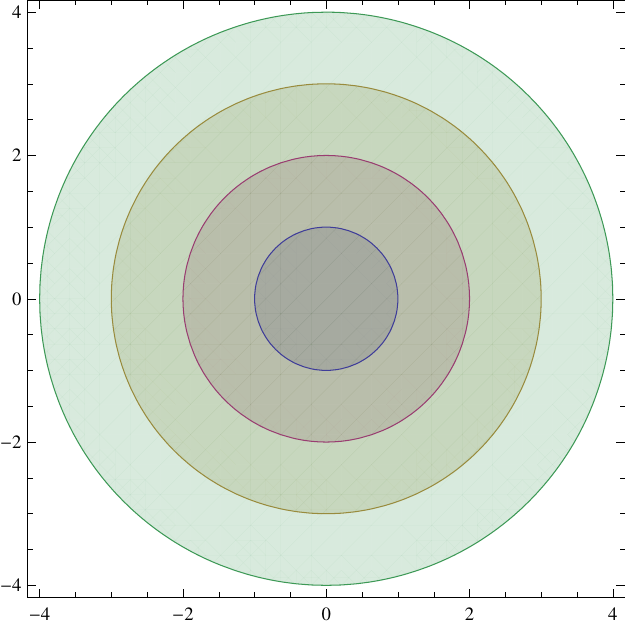}
    \end{overpic}
\end{tabular}
\caption{Left: The graph $G$. Right: $\mathcal{BW}_\Gamma(G)$.}
\end{figure}

\subsection{Targeting Specific Structural Sets}
Here we consider reducing graphs over specific structural sets in order to improve eigenvalue estimates when some structural feature of the graph is known. To do so consider $G=(V,E,\omega)$ where $V=\{v_1,\dots,v_n\}$.

If the sets $\mathcal{BW}_\Gamma(G)_i$ for $1\leq i\leq n$ are known or can be estimated by some structural knowledge of $G$ then it is possible to make decisions on which structural sets to reduce over. That is, it may be possible to identify structural sets $\mathcal{V}\subset V$ such that $v_i\notin\mathcal{V}$ and
$$\partial\mathcal{BW}_\Gamma(G)_i\nsubseteq\bigcup_{j\neq i}\mathcal{BW}_\Gamma(G)_j.$$
If this can be done, theorem \ref{theorem8} implies that a strictly better estimate of $\sigma(G)$ can be achieved by reducing over $\mathcal{V}$.

For example consider the graph $G=(V,E,\omega)$ in the left hand side of figure 10 where $V=\{v_1,\dots,v_n\}$ for some $n>5$. If it is known for instance that $G$ is a simple graph such that $d(v_1)=4$, $d(v_2)=d(v_3)=d(v_4)=d(v_5)=3$ and $d(v_i)\in\{0,1,2,3\}$ for all $6\leq i\leq n$ then the sets $\mathcal{BW}_\Gamma(G)_i$ are each discs of radius either 0,1,2,3 or 4 (see right hand side of figure 10). Moreover, as
$$\partial\mathcal{BW}_\Gamma(G)_1\nsubseteq\bigcup_{i=2}^n\mathcal{BW}_\Gamma(G)_i=\{\lambda\in\mathbb{C}:|\lambda|=4\}$$ then theorem \ref{theorem8} implies that $\mathcal{R}_{V\setminus\{v_1\}}(G)$ has a strictly smaller Gershgorin-type region than does $G$ which can be seen in figure 11. Considering the fact that $n$ may be quite large this example is intended to illustrate that eigenvalues estimates can be improved with a minimal amount of effort if some simple structural feature(s) of the graph are known.

However, it should be noted that as a graph is reduced its weights can contain increasingly larger powers of $\lambda$. Hence, the more a graph is reduced the more complicated it can become to compute the eigenvalue regions associated to it. Fortunately, there is a fairly simple bound for how large these powers of $\lambda$ can become.

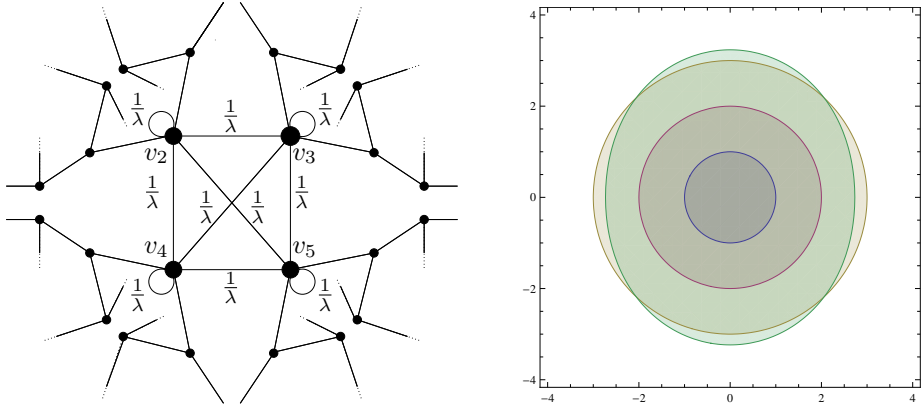
\begin{figure}
\begin{tabular}{cc}
    \begin{overpic}[scale=.35,angle=270]{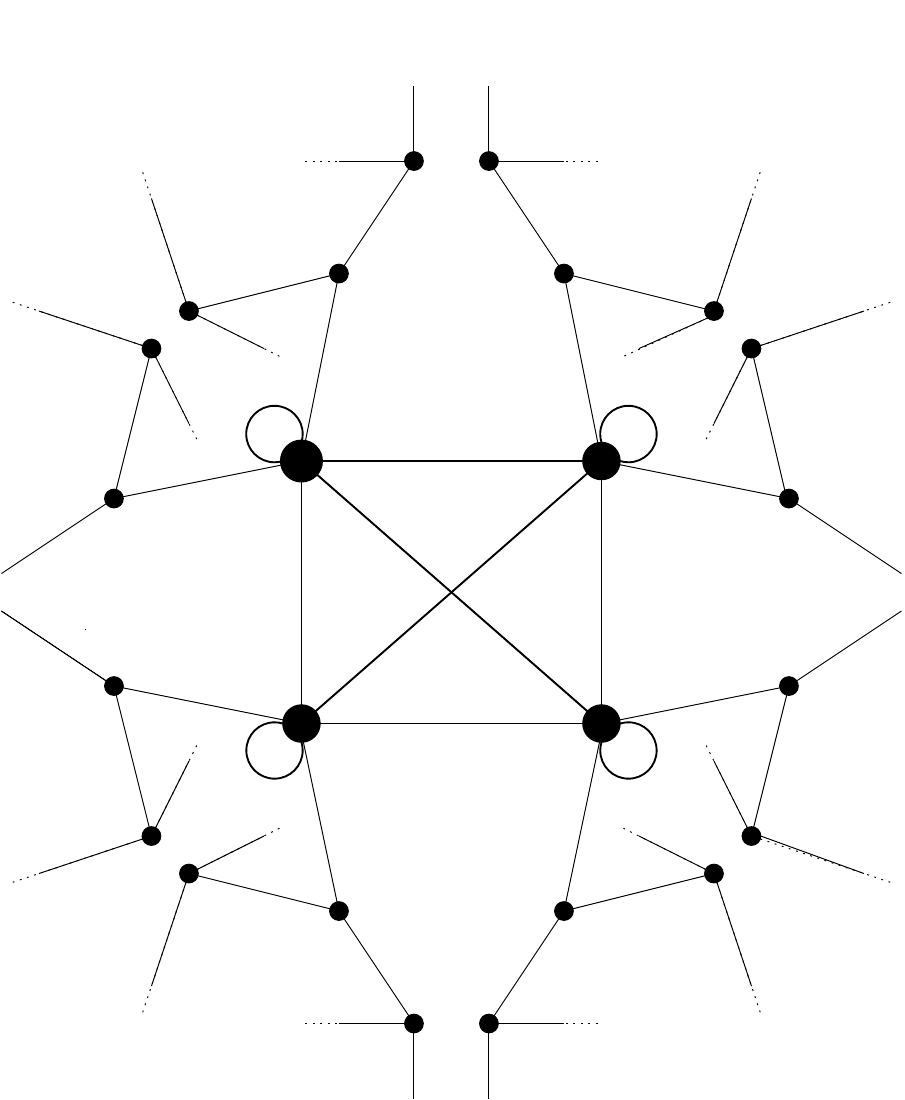}
    \put(25,59){$\frac{1}{\lambda}$}
    \put(44,57.5){$\frac{1}{\lambda}$}
    \put(63.25,59){$\frac{1}{\lambda}$}

    \put(28,42){$\frac{1}{\lambda}$}
    \put(39,39){$\frac{1}{\lambda}$}
    \put(49.5,39){$\frac{1}{\lambda}$}
    \put(58.5,42){$\frac{1}{\lambda}$}

    \put(25,21){$\frac{1}{\lambda}$}
    \put(44,22.5){$\frac{1}{\lambda}$}
    \put(63.25,21){$\frac{1}{\lambda}$}
    \put(58.5,31){$v_5$}
    \put(28.25,50){$v_2$}
    \put(58.5,50){$v_3$}
    \put(28.25,31){$v_4$}
    \end{overpic} &
    \begin{overpic}[scale=.5]{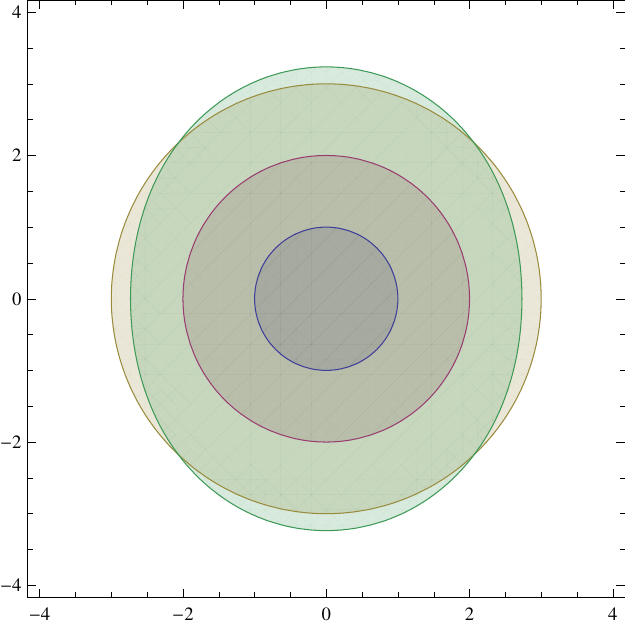}
    \end{overpic}
\end{tabular}
\caption{Left: $\mathcal{R}_{V\setminus\{v_1\}}(G)$. Right: $\mathcal{BW}_\Gamma(\mathcal{R}_{V\setminus\{v_1\}}(G))$.}
\end{figure}

Indeed, let $G=(V,E,\omega)$ such that $M(G)\in\mathbb{C}^{n\times n}$. For $\mathcal{V}\subset V$ let the entries $M\big(\mathcal{R}_{\mathcal{V}}[G]\big)_{ij}=p_{ij}/q_{ij}$
where $p_{ij},q_{ij}\in\mathbb{C}[\lambda]$. Then $$deg(p_{ij})\leq deg(q_{ij})\leq |\bar{\mathcal{V}}|<n.$$ For instance, $\mathcal{G}_1$ and $\mathcal{G}_2$ given in (\ref{last}) are examples of graphs that have been reduced from 5 to 3 and 2 vertices respectively. Hence, the largest power that $\lambda$ can be raised to in any entry of either $M(\mathcal{G}_1)$ or $M(\mathcal{G}_2)$ is 2 or 3 respectively. In fact, the largest power of $\lambda$ in $M(\mathcal{G}_2)$ is only 2.

That is, the Gershgorin, Brauer, and Brualdi-type regions become computationally harder to compute as a graph (equivalently matrix) is reduced but only marginally so. Moreover, this is offset to some degree by the fact that there are less Gershgorin and Brauer and often Brualdi type regions to compute for reduced graphs (matrices).

\section{Concluding Remarks}
The major goal of this paper is to demonstrate that isospectral graph reductions can be used to improve each of the classical eigenvalue estimates of Gershgorin, Brauer, Brualdi, and the more recent extension of Brualdi's theorem by Varga. Of major importance is the fact that these graph reductions are general enough that this process can be applied to any graph with complex valued weights (or equivalently matrices with complex valued entries). Hence, the aforementioned eigenvalue estimates of all matrices in $\mathbb{C}^{n\times n}$ can be improved via our process of isospectral graph reduction. Additionally, this process is sufficiently flexible to improve such eigenvalue estimates to whatever degree is desired.

Aside from this, the associated matrix reductions do not seem to require much computational effort. In fact, it may even be the case that our reduction method is sometimes computationally more feasible than standard methods of computing spectral properties. With regard to such questions, the computational complexity of our approach and its potential for computational improvements in calculating eigenvalues will be addressed in future publications.

\section{Acknowledgments}
This work was partially supported by the NSF grant DMS-0900945 and the Humboldt Foundation.

\end{document}